\documentclass[11pt]{article}

\usepackage{amsfonts,amssymb,amsmath,amsthm,mathtools}
\usepackage{url}
\usepackage{enumerate}
\usepackage[T1]{fontenc}
\usepackage{mathrsfs}
\usepackage{psfrag}
\usepackage{pst-all}
\usepackage[dvips]{graphicx}
\usepackage{xcolor}
\usepackage{tikz}
 \usepackage{soul}

\usepackage{cases}
\usepackage{caption, subcaption}
\usepackage{enumitem}
\usepackage[sc]{mathpazo}
\usepackage[utf8]{inputenc}

\definecolor{myPurp}{RGB}{105,73,148}
\definecolor{myRed}{RGB}{228,37,53}
\definecolor{myCyan}{RGB}{0,95,114}

\usepackage{hyperref}
\hypersetup{
     colorlinks=true,
     linkcolor=myRed,
     citecolor = myPurp,      
     urlcolor= myCyan,
     }

\usepackage{geometry}
\newtheorem{theorem}{Theorem}[section]

\newtheorem{lemma}[theorem]{Lemma}
\newtheorem{proposition}[theorem]{Proposition}
\newtheorem{corollary}[theorem]{Corollary}

\newtheorem{definition}[theorem]{Definition}
\newtheorem{remark}[theorem]{Remark}
\numberwithin{equation}{section}

\newcommand{\supp}{\textsc{supp}}

\newcommand{\psl}{\operatorname{PSL}_2(\R)}
\renewcommand{\sl}{\operatorname{SL}_2(\R)}
\renewcommand{\so}{\operatorname{SO}_2(\R)}

\renewcommand{\a}{\textsc{area}}
\newcommand{\A}{\mathbf{A}}
\newcommand{\Id}{\mathrm{Id}}
\newcommand{\p}{\mathrm{per}}

\newcommand{\diam}{\mathrm{diam}}
\newcommand{\K}{\mathcal{K}}

\renewcommand{\H}{\mathbb{H}}
\newcommand{\R}{\mathbb{R}}

\newcommand{\Hdim}{\operatorname{Hdim}}

\newcommand{\Ec}{\mathcal{E}_c^\pi}
\newcommand{\CH}{\operatorname{CH}}
\renewcommand{\SS}{\mathcal{S}}
\renewcommand{\o}{\mathbf{1}}
\newcommand{\1}{\mathbf{1}}

\newcommand{\E}{\mathcal{E}}

\DeclareMathOperator{\Stab}{Stab}
\DeclareMathOperator{\Isom}{Isom}

\newcommand{\qeda}{\hfill \ensuremath{\Box}}

\usepackage{array,xtab,ragged2e}
\newlength\mylengtha
\newlength\mylengthb
\newcolumntype{P}[1]{>{\RaggedRight}p{#1}}
\title{\bf An explicit exotic representation of a rank-one simple Lie group via convex bodies}
\author{François Fillastre, Yusen Long, David Xu}
\date{}

\begin{document}

\maketitle

\begin{abstract}
In \cite{DP}, Delzant and Py showed that there exist continuous irreducible isometric actions of \(\psl\) on the infinite-dimensional hyperbolic space \(\H^\infty\). Such continuous irreducible actions do not exist on the hyperbolic  spaces \(\H^n\) when \(n>2\) and their associated embeddings \(\H^2 \to \H^\infty\) given by the orbit maps were later called \emph{exotic} by Monod and Py in \cite{MP}. In this article, we produce a continuous and irreducible representation  of \(\psl\to \Isom(\H^\infty)\) using the hyperbolic model for convex bodies introduced in \cite{DF}. This yields a convex cocompact \(\psl\)-action on the infinite-dimensional hyperbolic space \(\H^\infty\), of which the compact quotient over the minimal \(\psl\)-invariant convex set is homeomorphic to the 2-dimensional oriented Banach--Mazur compactum. Moreover, we study the geometry of one of its orbit maps and compute the Hausdorff dimension of the limit set of this representation.
\end{abstract}

\tableofcontents

\section{Introduction}

Let \(\H^n\) be the \(n\)-dimensional real hyperbolic space for \(n\geq 1\). With a theorem of Karpelevich and Mostow \cite{karpelevic1953surfaces,mostow1955some} (or see \cite{boubel2004isometric} for a geometric proof in the hyperbolic cases), one can conclude the following fact: {\it any continuous isometric action of the group \(\Isom(\H^n)\) on \(\H^m\), for \(0<n\leq m<\infty\), preserves the image of a totally geodesic embedding \(\H^n\hookrightarrow\H^m\,.\)} However, this conclusion no longer holds if \(m=\infty\). 

In \cite{BIM}, Burger, Iozzi and Monod first constructed a family, parametrised by \(0<t<1\), of continuous representations \(\rho_t^T\colon\mathrm{Aut}(T)\to \Isom(\H^\infty)\), where \(T\) is a separable simplicial tree. This can be seen as realised by an \(\mathrm{Aut}(T)\)-equivariant quasi-isometric embedding \(T\hookrightarrow\H^\infty\). More than a decade later, this construction was generalised to any real tree by \cite[\S 13.1]{DSU} and \cite{monod}, thus also providing a family of continuous representations \(\Isom(\H^1)\simeq\Isom(\R)\to\Isom(\H^\infty)\) for \(0<t<1\), which do not preserve a bi-infinite geodesic line in \(\H^\infty\). Similar phenomena for the hyperbolic plane \(\H^2\) have also been observed in \cite{DP}, where Delzant and Py built a family of continuous representations \(\rho_t\colon\psl\to\Isom(\H^\infty)\) indexed by \(0<t<1\) that are irreducible, \emph{i.e.} the \(\rho_t\) do not preserve any proper hyperbolic subspace of \(\H^\infty\) and do not preserve any point on the boundary. Their construction was then generalised by Monod and Py in \cite{MP} to obtain continuous irreducible representations \(\Isom(\H^n)\to \Isom(\H^\infty)\). In the sequel, we will call a continuous irreducible representation of this kind an {\it exotic representation}. Later, by introducing the notion of {\it kernel of hyperbolic type}, Monod and Py were able to produce exotic representations \(\Isom(\H^\infty)\to \Isom(\H^\infty)\) \cite{MP2}. Moreover, in \cite{MP, MP2}, it is shown that the examples above actually give the complete list of exotic representations \(\Isom(\H^n)\to \Isom(\H^\infty)\) for \(2\leq n\leq \infty\) up to conjugacy in \(\Isom(\H^\infty)\). 


The main topic of this article will be an explicit exotic representation
\[\Isom^+(\H^2)\simeq \psl\to \Isom(\H^\infty)~.\]
In \cite{DP}, Delzant and Py studied the \(\psl\)-action on \(L^2(\SS^1)\), where here \(\SS^1\) is identified with the boundary at infinity of the hyperbolic plane (compare with our action in Section~\ref{sec:the psl action}). They defined a family of Lorentzian forms, which are invariant under the \(\psl\)-action, that allow them to build a hyperbolic space of infinite dimension and induce a family of irreducible representations \(\psl\to\Isom(\H^\infty)\). See \S\ref{section: comparison} for the construction of one of such representations.

In the present  article, we will present another construction of an  exotic representation \[\rho:\psl\to \Isom(\H^\infty)\] that is different from  the Delzant--Py construction or the Monod--Py construction (see \S\ref{section: comparison}). It is based on a hyperbolic model for convex bodies that was first introduced by Debin and the first author \cite{DF} and later studied by the second author \cite{long} for infinite-dimensional cases. It will appear that our \(\rho\) is conjugate to one of the representations they introduced. Similarly to \cite{DP}, our representation defines actually a  representation of \(\Isom(\H^2)\), see Remark~\ref{rep}.

 Let us briefly explain our construction of \(\rho\).  The area of symmetric convex bodies in the Euclidean plane can be polarised to obtain the so-called \emph{mixed area},  that satisfies the \emph{Minkowski inequality}, which resembles  the reversed Cauchy--Schwarz inequality, see Remark~\ref{remark area}. Identifying convex bodies with their support functions, the mixed area can be extended to a bilinear form over a Sobolev space, and the Poincaré--Wirtinger inequality says that it has a Lorentzian signature (Lemma~\ref{lem encadrement A}). In this model, the plane symmetric convex bodies with given constant area are naturally identified to a subset \(C_\K\) of an infinite-dimensional hyperbolic space \(\H^\infty\). In particular, this gives 
a map from  the space \(\E_c^\pi\) of symmetric ellipses with fixed area \(\pi\) to \(\H^\infty\). The space \(\E_c^\pi\) is the \(\psl\) orbit of the centred unit disc, hence we may identify \(\E_c^\pi\)  and \(\H^2\) to obtain a map  \[\iota\colon \H^2 \simeq \E^\pi_c \hookrightarrow \H^\infty\] that enjoys some striking geometric properties:
\begin{theorem}\label{thm principal}
    The map \(\iota:\H^2\to\H^\infty\) satisfies the following properties: 
    \begin{enumerate}[label=(T\arabic*), topsep=0pt, itemsep=-1ex, partopsep=1ex, parsep=1ex ]
      \item\label{T1} The image of \(\iota\) is an embedded minimal surface with constant curvature \(-\frac{8}{3}\) (Proposition~\ref{prop : smooth}, Proposition~\ref{prop: curvature} and Proposition~\ref{prop: minimal surface}).
 \item\label{T2} The embedding \(\iota\) is \(\psl\)-equivariant and induces an exotic  representation (\emph{i.e.} continuous irreducible) \(\rho\colon\psl\to \Isom(\H^\infty)\) (Proposition~\ref{prop : irreducible}).
 \item\label{T3}  The map \(\iota\) is a quasi-isometric embedding of \(\H^2\) into \(\H^\infty\) (Proposition~\ref{prop: quasi iso}). More precisely, there exists a constant \(C\) such that for all \( x,y \in \H^2\),
$$\left|d_{\mathbb{H}^\infty}(\iota(x),\iota(y))-\frac{1}{2}d_{\mathbb{H}^2}(x,y)\right|\leq C~.$$
\item\label{TX}  The minimal non-empty \(\rho(\psl)\)-invariant closed convex subset of \(\H^\infty\) is \(C_\K\) (Lemma~\ref{lem:ck minimal}), and it is the closed convex hull of \(\iota(\H^2)\) in \(\H^\infty\) (Corollary~\ref{cor: CH}). 
    Its isometry group is isomorphic to \(\psl\)  (Proposition~\ref{isom Ck}), and
    the action is proper and cocompact (Theorem~\ref{prop BM}).
                   \end{enumerate}
\end{theorem}

Point  \ref{TX} says that the action of \(\rho(\psl)\) is \emph{convex cocompact}, \emph{i.e.} there exists a minimal convex set with Hausdorff compact quotient.
The properties above can be derived from \cite{DP,MP} (see below), but here we  obtain them as  direct consequences of classical results about convex bodies, or by rather elementary computations.

As an illustration, the fact that \(\iota(\H^2)\) is not totally geodesic in \(\H^\infty\)  can be explained heuristically in the following way: the image by \(\iota\) of the axis of a hyperbolic translation of \(\H^2\) is not a geodesic line in \(\H^\infty\), see Figure~\ref{fig:quasi geodesic}.

\begin{figure}[h!]
\begin{center}
\psfrag{P}{$P$}
\psfrag{P1}{$P'$}
\psfrag{S}{$S$}
\includegraphics[width=0.5\linewidth]{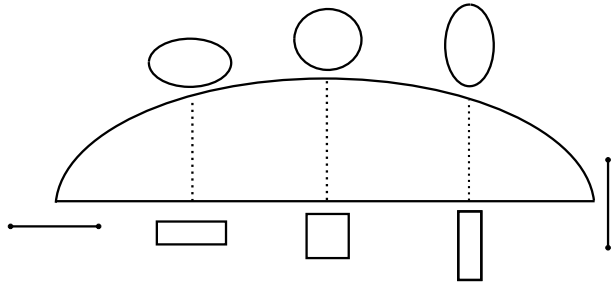}\caption{
The axis of a hyperbolic translation of \(\H^2\)  
acting on ellipses in \(\H^\infty\) is not a geodesic line: the \(\H^\infty\)-geodesic between the two endpoints at infinity of the geodesic line (the segments) is the convex combination of the segments, so a path of rectangles. Here the picture is in the Klein model of \(\H^\infty\).
}\label{fig:quasi geodesic}
\end{center}
\end{figure}

The point is that \(\Gamma\), the boundary at infinity of \(\iota(\H^2)\subset \H^\infty\), is naturally identified with the space of  unit symmetric segments, which is homeomorphic to the circle \(\mathbb{P}\R^2\simeq \SS^1\).

Moreover, we  obtain the following results that are new.
\begin{theorem}\label{prop:somme ellipse}
 \begin{enumerate}[label=(T\arabic*), topsep=0pt, itemsep=-1ex, partopsep=1ex, parsep=1ex]\setcounter{enumi}{4}
 The map \(\iota:\H^2\to\H^\infty\) satisfies the following additional properties: 
   \item\label{T4}   Any \(n+1\) distinct points in \(\iota(\H^2)\subset \H^\infty\) span an \(n\)-dimensional hyperbolic subspace of \(\H^\infty\) (Corollary~\ref{cor: k dimension}).
\item\label{T5} The Hausdorff dimension of \(\Gamma\) is \(2\) with respect to the visual distance of \(\partial_\infty \H^\infty\) (Corollary~\ref{cor:haus}, Proposition~\ref{prop: hausdorff dim general}).
\item \label{T6}  The quotient of \(C_\K/\psl\) is homeomorphic to the 2-dimensional oriented Banach--Mazur compactum (Theorem~\ref{prop BM}).
\end{enumerate}
\end{theorem}

Point \ref{T4} seems to be a folklore result, but we are not aware of any proof. Here it is an immediate consequence of our interpretation of \(\iota(\H^2)\) as a space of ellipses.
Indeed, it is well known that a finite sum of at least two non-homothetic ellipses is never an ellipse, see Proposition~\ref{prop: sum minkowski}. 
Point \ref{T4} implies that \(\iota(\H^2)\) is not totally geodesic in a stronger sense: a geodesic segment on \(\iota(\H^2)\) (for the induced metric) is never contained in any finite dimensional totally geodesic subspace of \(\H^\infty\). Note that by \ref{T3}, a geodesic line of \(\iota(\H^2)\) (for the induced metric) is a quasi-geodesic of \(\H^\infty\).

We will provide two different proofs of  \ref{T5}. One is via an explicit computation of the visual distance on \(\Gamma\), using our convex bodies interpretation, see Corollary~\ref{cor:haus}. The other one is the fact that, as the orbit map for the exotic representation \(\rho_t\) defined below  satisfies \eqref{eq:quasi t}, the Hausdorff dimension of the limit set is \(\frac{1}{t}\). This also holds for higher dimensional hyperbolic spaces \(\H^n\), see  Proposition~\ref{prop: hausdorff dim general}. The Hausdorff dimension of the limit set plays an important role in the study of group representations in \(\Isom(\H^n)\) and in Patterson--Sullivan theory. For finitely generated Fuchsian groups, \cite{beardon1971inequalities} introduced the \emph{critical exponent} as the infimum of \(s>0\) such that the associated \emph{Poincaré series} converges in dimension \(s\). Patterson then showed in \cite{patterson1976limit} that the Hausdorff dimension of the limit set of a Fuchsian group coincides with its critical exponent. This result was later extended to geometrically finite Kleinian groups by \cite{sullivan1979density}, and to any non-elementary Kleinian groups by Bishop and Jones in \cite{bishop1997hausdorff}. Moreover, \cite{bishop1997hausdorff} showed that the critical exponent is the Hausdorff dimension of the \emph{radial limit set} of these groups. The most general result that is known to us is in \cite[Theorem 1.2.1]{DSU}. In turn, Proposition~\ref{prop: hausdorff dim general} computes the critical exponent of the Monod--Py representation of any lattice in \(\Isom(\H^n)\) for \(n\geq 2\), as they are weakly discrete in \(\Isom(\H^\infty)\) (see \cite[Definition 5.2.1]{DSU} for definition) and its limit set is radial (see \cite[Definition 7.1.2]{DSU} for definition). In particular, the critical exponent of \(\rho(\mathrm{PSL}_2(\mathbb{Z}))\) is \(2\).

In  \ref{T6}, the {\it \(2\)-dimensional Banach--Mazur compactum} is the metric space of isometry classes of 2-dimensional Banach spaces, see \cite{BM3,BM2,BM1,BM4}. It can also be identified with the quotient of symmetric convex bodies with fixed area by \(\mathrm{GL}_2(\R)\). In turn, its double cover is naturally homeomorphic to the quotient of the space of these convex bodies up to orientation preserving linear maps, {\it i.e.} the {\it \(2\)-dimensional oriented Banach--Mazur compactum} introduced in \cite{BM3}.

We recall from \cite{MP} that the complete list of exotic representations
\[\rho_t\colon\Isom(\H^2)\to \Isom(\H^\infty)\]
is parametrised by \(0<t<1\). The parameter \(t\)  is characterised by the fact that, if  \(f_t \) is the orbit map  of the unique fixed point by \(\rho_t(\so)\), then \(f_t\) is a quasi-isometric embedding: 
\begin{equation}\label{eq:quasi t}\left| d_{\H^\infty}\big( f_t(x), f_t(y)\big)-td_{\H^2}(x,y)\big)\right|\leq D\end{equation}
for some \(D>0\) and for all \(x,y\in \H^2\).

A natural question for the representation arising from the hyperbolic model for convex bodies is: {\it for which \(0<t<1\) is \(\rho\) conjugate to \(\rho_t\) in \(\Isom(\H^\infty)\)?}

\begin{corollary}\label{cor:conjug}
Our representation $\rho$ is conjugate to $\rho_{\frac{1}{2}}$ in \(\Isom(\mathbb{H}^\infty)\).
\end{corollary}

This result  is an immediate consequence of \ref{T3} of Theorem~\ref{thm principal} by comparing to \eqref{eq:quasi t}. It  also follows from \ref{T1} of Theorem~\ref{thm principal} after the following observation: 
 \[\frac{3}{8}=\frac{t\left(t+2-1\right)}{2}\]
 holds for \(t=1/2\), see  Theorem~C(2) in \cite{MP}.

From  different aspects, Section~\ref{last sec} is dedicated to distinguish  our construction.

In \S \ref{uniq embedding}, we first check that \(\iota\) is the only \(\rho\)-equivariant map from \(\H^2\) to \(\H^\infty\). We then prove that any continuous \(\rho\)-equivariant map that satisfies a convexity condition (see Proposition~\ref{prop extension supp}) from symmetric convex bodies of area \(\pi\) to \(\H^\infty\) is our map \(\supp\).

In \S\ref{4.2}, we recall some facts about \emph{kernels of (real) hyperbolic type}, introduced in \cite{MP2}, and explain how it is related to our construction. More precisely, this notion is an analogue to the notion of positive definite kernels for constructing embeddings into Hilbert spaces via the renowned Gelfand--Naimark--Segal construction ({\it GNS construction}), see \cite[Appendix C]{Bekka2008kazhdan}. Moreover, this notion should not be confused with the kernel of \emph{complex} hyperbolic type introduced in \cite{monod}, which was used to produce exotic complex hyperbolic representations \(\psl\to \Isom(\H^\infty_\mathbb{C})\) by Ruiz Stolowicz in \cite{RuizStolowicz}. As a byproduct of interpreting the mixed area as a kernel of hyperbolic type, we obtain an extension of the Minkowski inequality for plane convex bodies, see Proposition~\ref{prop mink neg}, which also holds for higher dimensional convex bodies with their second intrinsic volume.

There are different ways to embed \(\H^2\) in a \(\psl\)-equivariant way  into \(\mathbb{H}^\infty\), and these embeddings all have quasi-isometric images. In \S\ref{section: comparison}, we first note that they do not necessarily reduce to an isometry. However, if the associated representations are irreducible, then \cite[Theorem B]{MP} states that the orbits of the unique  \(\so\)-fixed point are isometric. In particular, \(\iota(\H^2)\) and \(f_{\frac 1 2}(\H^2)\) are isometric. We verify this fact by a direct calculation at the end of \S\ref{section: comparison}.

Two further questions emerged as a natural continuation of what has been studied in the present article.

The first one is: \emph{can we construct an exotic representation \(\rho\colon \Isom(\H^n)\to \Isom(\H^\infty)\) using convex bodies?} The answer is no, and  the present construction is a dimensional accident, as we now explain. 

The hyperbolic embeddings of convex bodies does not restrict to plane convex bodies. The construction in the present article is a particular case of hyperbolic embedding of convex bodies in any \(\R^n\) \cite{DF}, or even in a separable Hilbert space \cite{long}. In particular, the pivotal Lemma~\ref{lem encadrement A}, based on the Poincaré–Wirtinger Inequality, is a particular case of properties of the spherical Laplacian \cite[Lemma 2.3]{DF} or of Malliavin calculus in infinite dimension \cite[Section 4]{long}.
In any case, the support function map provides a convex embedding from convex bodies to infinite dimensional hyperbolic space. It is not restricted to symmetric convex bodies, but to convex bodies up to translations, or, that is equivalent, to convex bodies with Steiner point at the origin.

One may wonder how the volume-preserving action of $\operatorname{SL}_n(\R)$ acts over the ellipsoids in $\R^n$, in order to derive a representation of $\operatorname{SL}_n(\R)$ into $\Isom(\H^\infty)$. Actually, for \(n\geq 3\), the Lie group $\operatorname{SL}_n(\R)$ has Kazhdan's Property (T), which prevents it from acting non-elementarily by isometries on any real hyperbolic space, see for example \cite[\S1.4 and \S2.6]{Bekka2008kazhdan}. Also, the hyperbolic embedding of convex bodies in higher dimensions is based on the intrinsic area, which coincides with the volume only for the plane. For example, for convex polyhedra, it is, up to a dimensional constant, the sum of the 2-dimensional volume of the faces of dimension 2, and this is not preserved by \(\operatorname{SL}_n(\R)\) for \(n>2\). So, in higher dimension, our construction provides an embedding of the symmetric space of \(\operatorname{SL}_n(\R)\) into \(\H^\infty\), but the action is not isometric. We also remark that in dimension \(2\), the hyperbolic plane is homothetic to the symmetric space \(\sl/\so\) endowed with the Killing metric.


Also, in higher dimensions, the boundary at infinity is always given by the symmetric segments. But the convex hull of this boundary at infinity  gives all the symmetric convex bodies only for the plane (Lemma~\ref{lem Z=C}), but not for higher dimensions.

 The second question  is: \emph{can we describe \(\rho_t\) for \(t\not= \frac 1 2\) using convex bodies?} As the form \(\A\) is a kernel of hyperbolic type on the support functions of plane symmetric convex bodies with positive area, one can immediately deduce from \cite[Theorem~3.10]{MP2} that for \(0< s\leq 1\), the power \((\A(\cdot,\cdot))^s\) remains a kernel of hyperbolic type on the same set, which then yields again an exotic representation \(\psl\to \Isom(\H^\infty)\) conjugate to \(\rho_{\frac s 2}\). Hence, the representation \(\rho_{\frac s2}\) for \(0<s\leq 1\) also enjoys properties similar to Theorem~\ref{thm principal} and Theorem~\ref{prop:somme ellipse}. Actually,  Theorem~\ref{thm principal} is already known for any \(\rho_t\) \((0<t<1)\), while  \ref{T5} is shown to be true for any \(\rho_t\) in the present article. However, \cite[Theorem~3.10]{MP2} does not apply for \(s>1\), this means that our construction cannot go higher than the parameter \(\frac 1 2\) for the time being.

\subsection*{Acknowledgement}
The authors thank Nicolas Monod for sharing his insights on the connection between the exotic representation and the hyperbolic model of convex bodies. They are grateful to Gilles Courtois, Bruno Duchesne, Antonin Guilloux, Gye-Seon Lee, and Pierre Py for their interest and for many helpful discussions. The authors also thank the anonymous referee for pointing out some gaps in the proofs and for numerous remarks.

The second author also thanks the hospitality of Korea Institute for Advanced Study and University of Montpellier, where this work was partially carried out during his visit.

The second author is sponsored by the ANR Tremplin ERC-Starting Grant MAGIC (ANR-23-TERC-0007) and by the ANR project Grant GALS (ANR-23-CE40-0001). The third author is supported by the KIAS Individual Grant (MG100801) at Korea Institute for Advanced Study. The second and the third authors are also partially supported by the ``PHC Star'' programme (project number PHC 50166PH and RS-2023-00259480), funded by the French Ministry for Europe and Foreign Affairs, the French Ministry for Higher Education and Research, and the National Research Foundation (NRF) of Korea (Ministry of Science and ICT).



\section{Hyperbolic model for plane convex bodies}\label{sec : hyp}

In this section, we will first review the hyperbolic model for convex bodies in \cite{DF,long} by going through the tools that were used in the construction. In reviewing this construction, we will also mention some conclusions that can be drawn directly from it.

\subsection{Hyperbolic space}

First, let us recall the notion of an infinite-dimensional real hyperbolic space using the hyperboloid model. Let \(\mathcal{H}\) be a Hilbert space over \(\mathbb{R}\) carrying an inner product \(\langle \cdot,\cdot\rangle_\mathcal{H}\). 
For \(x\in \mathbb{R}\oplus\mathcal{H}\), denote \(x=x_0+x_\mathcal{H}\).
Define the quadratic form \(B\colon (\mathbb{R}\oplus\mathcal{H})\times (\mathbb{R}\oplus\mathcal{H})\to \mathbb{R}\) as
\begin{align}\label{eq : lorentzian}
B\big(x,y\big)\coloneqq x_0 y_0-\langle x_\mathcal{H},y_\mathcal{H}\rangle_\mathcal{H}~.
\end{align}

The {\it hyperboloid model} for the real hyperbolic space is given by
\[\mathbb{H}^\infty\coloneqq\left\{x\in \mathbb{R}\oplus\mathcal{H}\mid x_0>0\text{ and }B(x,x)=1\right\}\]
and is equipped with a distance \(d_{\mathbb{H}^\infty}\colon\mathbb{H}^\infty\times \mathbb{H}^\infty\to \mathbb{R}_+\) defined as
\[d_{\mathbb{H}^\infty}(x,y) =\cosh^{-1}\left(B(x,y)\right)~.\]
Indeed, such a quadratic form \(B\) satisfies the \emph{reversed Cauchy--Schwarz inequality}: for \( (x_0,x_\mathcal{H}),(y_0,y_\mathcal{H})\) in \(\R^+\oplus\mathcal{H}\) such that \(B\big(x,x\big)>0, B \big(y,y\big)>0\), then
\begin{equation}\label{eq:RCS}~
B\big(x,y)^2\geq B\big(x,x\big)B\big(y,y\big)
\end{equation}
with equality if and only if the two vectors are proportional. In particular, 
if \(x,y\in \mathbb{H}^\infty\), then \(B(x,y)\geq 1\), with equality if and only if \(x=y\).

This space is a complete Riemannian manifold of constant sectional curvature equal to \(-1\). It is therefore a \(\mathrm{CAT}(-1)\) space and admits a boundary at infinity \(\partial_\infty\mathbb{H}^\infty\) defined as the set of equivalence classes of geodesic rays, two geodesic rays \(c_1,c_2 : [0,\infty) \to \mathbb{H}^\infty\) being equivalent if and only if they are asymptotic, {\it i.e.} there exists a constant \(K \geq 0\) such that \(d_{\mathbb{H}^\infty}(c_1(t),c_2(t)) \leq K\) for all \(t\geq 0\). We refer to \cite[Chapter II.8]{BH} for the definition and properties of the boundary at infinity for general complete \(\mathrm{CAT(0)}\) spaces.

Alternatively, we may also define \(\mathbb{H}^\infty\) as the set of positive lines in \(\mathbb{R}\oplus\mathcal{H}\),
\[\mathbb{H}^\infty\coloneqq\left\{[x] \in \mathrm{P}(\mathbb{R}\oplus\mathcal{H}) \mid B(x,x) > 0\right\}\]
where \(\mathrm{P}(\mathbb{R}\oplus\mathcal{H})\) is the projectivization of \(\mathbb{R}\oplus\mathcal{H}\). With this definition, the boundary of \(\mathbb{H}^\infty\) as a subset of the topological space \(\mathrm{P}(\mathbb{R}\oplus\mathcal{H})\) is the set of isotropic lines in \(\mathbb{R}\oplus\mathcal{H}\), \(\left\{[x] \in \mathrm{P}(\mathbb{R}\oplus\mathcal{H}) \mid B(x,x) = 0\right\}\). It is known that this (topological) boundary coincides with the boundary at infinity \(\partial_\infty\mathbb{H}^\infty\). Indeed, being \(\mathrm{CAT(-1)}\) implies that \(\partial_\infty\mathbb{H}^\infty\) can be identified with the Gromov boundary of \(\mathbb{H}^\infty\) by \cite[Proposition 4.4.4]{DSU}. Besides, the identification between the Gromov boundary and the topological boundary was proved for example in \cite[Proposition 3.5.3]{DSU}. The descriptions of \(\partial_\infty\mathbb{H}^\infty\) as the isotropic cone in \(\mathbb{R}\oplus\mathcal{H}\) and as the set of classes of geodesic rays in \(\mathbb{H}^\infty\) will both be useful in the remaining of this article.

Moreover, we will define \(\dim \mathbb{H}^\infty\coloneqq \dim \mathcal{H}\), because it is clear from the hyperboloid model that \(\mathbb{H}^\infty\) is a \(\mathcal{H}\)-manifold, {\it i.e.} it is locally homeomorphic to \(\mathcal{H}\). Given two Hilbert spaces of the same dimension, one can construct an isomorphism between them via a bijection between their respective orthonormal basis. Consequently, we can conclude that two hyperbolic spaces are isometric if and only if they have the same dimension, {\it i.e.} the hyperbolic space of a given dimension is unique up to isometry (see also for example \cite[Propositions 2.7 and 3.7]{BIM}). Hence, we are able to identify all infinite-dimensional separable real hyperbolic spaces, and in the present article we will use the notation \(\mathbb{H}^\infty\) for such a space up to isometry.

We end this subsection by remarking that by considering its intersections with finite-dimensional vector spaces, the distance \(d_{\H^\infty}\) is the one given by the induced Riemannian metric \(g_{\H^\infty}\) over \(\H^\infty\), that \(\mathbb{H}^\infty\) has isometries acting transitively \cite[Lemma 3.6]{BIM}, and that it is the same as in the finite-dimensional case where distance isometries are Riemannian isometries \cite{Garrido-Jaramillo-Rangel}. We refer to \cite{DSU} for more details. 

\subsection{Space of periodic even functions}

Let \(H^1_{\operatorname{even}}\) be the Hilbert subspace of even functions of the Sobolev space on the unit circle \(\SS^1\), endowed with the Sobolev norm \(\|\cdot\|_{H^1}\) given by:
\[\|h\|^2_{H^1}=\|h\|^2_{L^2}+\|h'\|^2_{L^2}~.\]
By Stone--Weierstra{\ss} Theorem, we can approximate these functions by polynomials on the unit circle, so \(H^1_{\operatorname{even}}\) is separable and thus of countably infinite dimension. Parametrising the circle by \((\cos(\theta),\sin(\theta))\) for \(\theta\in[0,2\pi)\), the space \(H^1_{\operatorname{even}}\) will then consist of \(\pi\)-periodic functions on \([0,2\pi)\). We then endow \(H^1_{\operatorname{even}}\) with the bilinear form
\begin{equation}\label{eq: def A}
\A(h_1,h_2)=\frac{1}{2\pi}\int_0^{2\pi}(h_1h_2 -h'_1h'_2)~,
\end{equation}
and denote
\[\A(h)\coloneqq\A(h,h)=\frac{1}{2\pi}\|h\|^2_{L^2} - \frac{1}{2\pi}\| h'\|^2_{L^2}~,\]
where \(h'\) is the derivative of the function \(h\) with respect to \(\theta\).
We will denote by \(\o\) the constant function equal to one. Hence
\[\A(\o)=1~.\]

Let \(L_0\) be the line of constant functions in \(H^1_{\operatorname{even}}\), and \(L_0^\bot\) be its  \(L^2\)-orthogonal. 
\begin{lemma}\label{lem encadrement A}
 \begin{equation}\label{eq:encadrement A}
\forall h\in L_0^\bot,  \frac{3}{10\pi}\|h\|_{H^1}^2\leq -\A(h) \leq \frac{1}{2\pi}\| h\|^2_{H^1}
 \end{equation}
\end{lemma}
\begin{proof}
The right-hand side inequality in \eqref{eq:encadrement A} is immediate as \(-2\pi\A(h)\leq \|h'\|^2_{L^2}\leq \|h\|_{H^1}^2\).

As the unique even function on the circle that is the restriction of a linear form of the plane is the zero function, by Poincaré--Wirtinger Inequality for integral zero cases (see \emph{e.g.} \cite[Section 4.4]{groemer}), 
 \begin{equation}\label{eq:Wir1}
\forall h\in L_0^\bot, \|h\|_{L^2}^2\leq \frac{1}{4}\|h'\|_{L^2}^2
 \end{equation}
with equality if and only if \(h=0\). 
From \eqref{eq:Wir1}, we obtain \(-2\pi\A(h)\geq 3\|h\|^2_{L^2}\) but also \(-8\pi\A(h)\geq 3\|h'\|^2_{L^2}\), and left-hand side inequality in  \eqref{eq:encadrement A} is obtained by summing those two inequalities.
\end{proof}

For any \(h\in H^1_{\operatorname{even}}\), let us denote  
 \begin{equation}\label{eq:pi0}\pi_0(h)=\frac{1}{2\pi}\int_0^{2\pi}h=\A(h,\o)~.\end{equation} 
Clearly, the two functions \(h-\pi_0(h)\o\) and \(\o\) are \(\A\)-orthogonal, hence \(\pi_0(h)\o\) is the orthogonal projection of \(h\) onto  \(L_0\). For future reference, it is worth noting that
\begin{equation}\label{eq dev A}
\A(h)=\A\left(h-\pi_0(h)\o\right) +\pi_0(h)^2~.
\end{equation}

Besides, we can provide \(H^1_{\operatorname{even}}\) with coordinates such that \eqref{eq : lorentzian} holds, with the Hilbert space \(\mathcal{H}\) being \(L_0^\bot\) endowed with \(-\A\). 
Since \(H^1_{\operatorname{even}}\) is of countable infinite dimension, we have
\[\H^\infty\coloneqq\Big\{h\in H^1_{\operatorname{even}} \mid \A(h)=1,\ \pi_0(h)>0\Big\} \]
 endowed with the distance 
 \begin{equation}\label{eq:def dinfty}
 d_{\H^\infty}(h_1,h_2)\coloneqq\cosh^{-1}(\A(h_1,h_2))~.
 \end{equation}
 




Let us present some basic facts about the topology of \(\H^\infty\). Let us first make a convenient remark.

\begin{remark}\label{rem cv dH A}
By the definition \eqref{eq:def dinfty} of \(d_{\H^\infty}\),  \(d_{\H^\infty}(h_i,h)\to 0\) is equivalent to \(\A(h_i,h)\to 1\).  But  \(h_i,h \in \H^\infty\), hence
\(\A(h_i-h)=2-2\A(h_i,h)\), so \(d_{\H^\infty}(h_i,h)\to 0\) is equivalent to \(\A(h_i-h)\to 0\).
\end{remark}

\begin{lemma}\label{lem: cont proj}
The map \(\pi_0\big|_{\H^\infty}:(\H^\infty,d_{\H^\infty})\to \R\) is continuous.
\end{lemma}

\begin{proof}
Let \(h_i\to h\) in \(\H^\infty\). For any \(i\),  consider the vector space spanned by \(L_0,h,h_i\) (that may be of dimension 1, 2 or 3). The intersection with \(\H^\infty\) gives a finite-dimensional hyperbolic space and the distance between \(h_i\) and \(h\) is preserved, as well as their image for \(\pi_0\). The distance between those images can then be explicitly computed.
\end{proof}

\begin{lemma}\label{lem:meme topo}
\(d_{\H^\infty}\) and \(\|\cdot  \|_{H^1}\) define the same topology on \(\H^\infty\).
\end{lemma}
\begin{proof}
Let \((h_i)_i\) be a sequence of \(\H^\infty\) and \(h\in \H^\infty\). We will use Remark~\ref{rem cv dH A}.

Let us suppose that \(h_i\to h\) for \(\|\cdot  \|_{H^1}\).
Hence \((h_i)_i\)
converges to \(h\) in \(L^2\) and \((h'_i)_i\)
converges to \(h'\) in \(L^2\), so it follows immediately from the definition of \(\A\) that \(\A(h_i-h)\to 0\). 

Conversely, let us suppose that \(h_i\to h\) for \(d_{\H^\infty}\), {\it i.e.} \(\A(v_i)\to 0\) with \(v_i=h_i-h\).
By Lemma~\ref{lem: cont proj}, \(\pi_0(v_i)\to 0\), 
hence \(\A\left(v_i-\pi_0(v_i)\o\right)\to 0\) by \eqref{eq dev A}, and by \eqref{eq:encadrement A}, 
\(\left\|v_i-\pi_0(v_i)\o\right\|_{H^1}\to 0\). Using again \(\pi_0(v_i)\to 0\), that gives \(\|v_i\|_{H^1}\to 0\).
\end{proof}

\begin{lemma}\label{lem bounded 1}
A subset of \(\H^\infty\) that is bounded for \(d_{\H^\infty}\) is also bounded for \(\|\cdot  \|_{H^1}\).
\end{lemma}

\begin{proof}
A subset of \(\H^\infty\) that is bounded for \(d_{\H^\infty}\) is contained in an open ball centred at \(\o\), so it suffices to prove that such a ball of radius \(c\) is bounded for \(\|\cdot \|_{H^1}\). So, let 
 \(h\in \H^\infty\) such that \(d_{\H^\infty}(h,\o)\leq c\).
By the definition of the distance, for such a \(h\), we have
\[ \pi_0(h)=\frac{1}{2\pi} \int_0^{2\pi}h=\A(h,\o)=\cosh(d_{\H^\infty}(h,\o))\leq \cosh(c)\eqqcolon k>1,\]
so by \eqref{eq dev A} and  \eqref{eq:encadrement A}, 
\[0<\A(h)<-\frac{3}{16\pi} \|h-\pi_0(h)\o\|^2_{H^1} + k^2
 \]
so
\[ \|h\|_{H^1}^2=\|h-\pi_0(h)\o\|_{H^1}^2+\|\pi_0(h)\o\|^2_{H^1}< Ck^2\]
for some constant \(C> 0\).
\end{proof}


\subsection{Space of symmetric convex bodies}

This construction of the infinite-dimensional hyperbolic space has a clear geometric interpretation, {\it i.e.} the symmetric plane convex bodies can be associated to even functions on \(\SS^1\). 
Recall that a convex body is called \emph{symmetric} if it is centrally symmetric with respect to the origin, {\it i.e.} $-K=K$.
The philosophy comes from the integral geometry and geometric probability theory, where geometric objects are identified with functions or random variables that describe their shapes.

In this section, we refer to \cite{groemer,schneider} for all the facts we mention about convex bodies. 

Let $\mathcal{K}_c$ be the collection of convex bodies in the plane with non-empty interior. Let $\overline{\mathcal{K}}_c$ be the set of symmetric convex bodies possibly with empty interior. 
Let $\langle \cdot,\cdot\rangle$ be the usual scalar product on $\R^2$. The \emph{support function} $\supp(K)\colon\SS^1\to \R$ of $K$ 
is defined by 
\begin{equation}\label{eq:def supp}\supp(K)(x)=\sup_{k\in K}\langle x,k\rangle~ .\end{equation}

Besides, the support functions of convex bodies in \(\R^2\) are exactly the functions over \(\SS^1\) whose one-homogeneous extensions over \(\R^2\) are convex. This provides strong regularity properties. See \cite{MM} for other characterisations. If \(K\) is contained in the centred ball of radius \(R\), then \(\supp(K)\) is \(R\)-Lipschitz. Also, \(K\) being symmetric is equivalent to \(\supp(K)(x)=\supp(-K)(x)=\supp(K)(-x)\) over \(\SS^1\).  Hence for a convex body \(K\in \overline{\mathcal{K}}_c\), its support function \(\supp(K)\) belongs to \(H^1_{\operatorname{even}}\).  

So we have a well-defined map
\begin{align*}
\supp\colon \overline{\K}_c &\to H^1_{\operatorname{even}}\\
K&\mapsto \supp(K)
\end{align*}
and the main interest of this map is that it preserves basic operations on convex bodies, namely it is additive, {\it i.e.} \begin{equation}\label{eq:addition supp}\supp(K)+\supp(K')=\supp(K+K')~,\end{equation} where
\[K+K'\coloneqq\{v+w\in\R^2 \mid v\in K, w\in K'\}\]
is the {\it Minkowski sum} of the two convex bodies. Also,
\begin{equation}\label{eq:lambda supp}\forall \lambda>0,~ \supp(\lambda K)=\lambda\supp(K)~.\end{equation}


We endow \(\overline{\mathcal{K}}_c\) with the subspace topology, which is the topology induced by the Hausdorff distance \(d_{\operatorname{Haus}}\). 
These classical results allow us to express the Hausdorff distance in a simple way.
\begin{lemma}\label{topo unif f s}
The map
\[\supp\colon (\overline{\K}_c,d_{\operatorname{Haus}}) \to (H^1_{\operatorname{even}},d_{L^\infty})\] is an isometry, where \(d_{L^\infty}\) is the distance defined by \(\|\cdot\|_{L^\infty}\).
\end{lemma}

From this and Arzel\`a–Ascoli's theorem follows the famous Blaschke selection theorem:

\begin{proposition}\label{prop:blaschke}
\((\overline{\K}_c,d_{\operatorname{Haus}})\) is a proper metric space, {\it i.e.} every bounded closed subset is compact. In turn it is locally compact and complete.
\end{proposition}

The regularity properties of the support functions implies the following strengthened version of Lemma~\ref{topo unif f s}.

\begin{lemma}\label{lem:supp cont H1}
The map
\[\supp\colon (\overline{\K}_c,d_{\operatorname{Haus}}) \to (H^1_{\operatorname{even}},\|\cdot\|_{H^1})\] is continuous.
\end{lemma}
\begin{proof}
 For a convergent sequence of convex bodies, by Lemma~\ref{topo unif f s}, their support functions converge for the uniform distance, and hence for the \(L^2\)-distance. Moreover, if \(K_i\to K\), then \(h_i\coloneqq\supp(K_i)\) are uniformly bounded, so by \cite[Lemma 2.2.1]{groemer}, their gradients are uniformly bounded and converge almost everywhere. This implies that they converge in \(L^2\), and as a result, \(h_i\) converge to \(h=\supp(K)\) in \(H^1\)-norm.
\end{proof}

\begin{lemma}\label{lem:K barre ferme}
\(\supp(\overline{\K}_c)\) is closed for  \(\|\cdot\|_{H^1}\).
\end{lemma}
\begin{proof}
Let \(h_n\in \supp(\overline{\K}_c)\) such that \(h_n\to h\) for \(\|\cdot\|_{H^1}\). Hence, it converges for the \(L^2\)-norm.
By \cite[Theorem 3]{vitale}, the space of support functions is closed for the \(L^2\)-norm, so \(h\) is a support function.
\end{proof}

Let \(\mathrm{S}\) be the set of centrally symmetric segments of the plane with positive length. Then
\begin{equation}\label{eq:intro S}\overline{\mathcal{K}}_c=\mathcal{K}_c\sqcup  \mathrm{S}\sqcup  \{0\}~.\end{equation}
Note that for \(K \in \mathcal{K}_c\), \(\supp(K)\) is positive, as the origin is contained in the interior of \(K\).

\begin{figure}[h!]
\begin{center}
\psfrag{K}{$\partial K$}
\psfrag{k}{$k$}
\psfrag{n}{$\binom{\cos\theta}{\sin\theta}$}
\psfrag{np}{$\binom{-\sin\theta}{\cos\theta}$}
\psfrag{h}{$h_K(\theta)$}
\psfrag{S}{$\SS^1$}
\psfrag{hp}{$h_K'(\theta)$}
\psfrag{l}{\(\ell\)}
\includegraphics[width=0.5\linewidth]{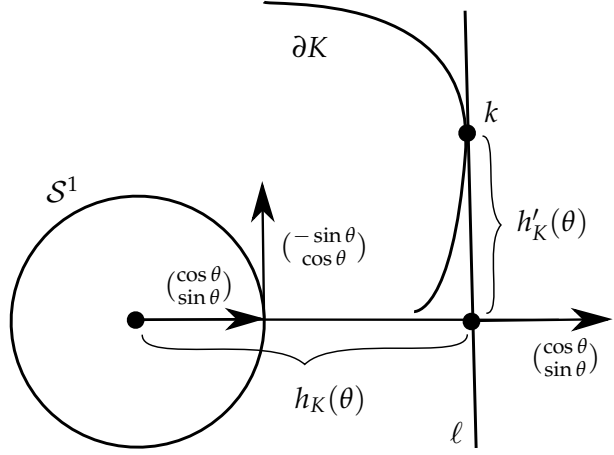}\caption{To Remark~\ref{remark area}. Parametrisation of the boundary of a convex body by the support function.}\label{fig:parametr}
\end{center}
\end{figure}

\begin{remark}\label{remark area}
Let \(K\) be a strictly convex body on the Euclidean plane whose boundary $\partial K$ is $C^2$. By strict convexity,  there is a tangent line $\ell$ orthogonal to $(\cos(\theta),\sin(\theta))$ which meets $\partial K$ at a unique point $c_K(\theta)$. This gives a parametrisation \(c_K:\SS^1\to \partial K\) of the boundary. As $h_K(\theta)\coloneqq\supp(K)(\theta)$ is exactly the distance from $\ell$ to the origin of $\R^2$, there exists a function $g$ such that 
$c_K(\theta)=h_K(\theta)\binom{\cos\theta}{\sin\theta} + g(\theta)\binom{-\sin\theta}{\cos\theta}$.  As $c_K'(\theta)$ is tangent to $\partial K$, we have $\langle c_K'(\theta), \binom{\cos\theta}{\sin\theta} \rangle=0$, that gives the following (see Figure~\ref{fig:parametr}):

\begin{equation}\label{eq:param courbe}c_K(\theta)=h_K(\theta)\binom{\cos\theta}{\sin\theta} + h_K'(\theta) \binom{-\sin\theta}{\cos\theta}~.\end{equation}

If \(\a(K)\) is the area of the convex body \(K\), the Green--Riemann formula and \eqref{eq:param courbe}  give
\begin{equation*}\label{eq:v2lisse}\a(K)=\frac{1}{2}\int_{\partial K} x \mbox{d}y-y\mbox{d}x=\frac{1}{2}\int_0^{2\pi} h_K(h_K+h_K'')~.\end{equation*}
As $h_K(0)=h_K(2\pi)$ and $h_K'(0)=h_K'(2\pi)$, 
$$0=\int_0^{2\pi} (h_Kh_K')'=\int_0^{2\pi}h_K'^2 + \int_0^{2\pi}h_Kh_K''~,$$
we obtain 
\begin{equation}\label{eq:area equality}
\a(K)=\pi \cdot \A(\supp(K))~.
\end{equation}
This formula holds for any convex body by approximation, and was first  proved by Blaschke  (\cite{blaschke}, see also  \cite[p.298]{schneider} and the references therein).

The constant \(\pi\) is easy to check:  if \(D\) is the centred unit disc of the Euclidean plane, then \(\o=\supp(D)\), and 
 \(\A(\o)=1\) but \(\a(D)=\pi\).

Up to this multiplicative constant,  \(\A(\supp(K_1),\supp(K_2))\) corresponds to the \emph{mixed area} of the convex bodies \(K_1\) and \(K_2\).  This has many geometrical implications, for example  the reversed Cauchy--Schwarz inequality \eqref{eq:RCS} for the Lorentzian bilinear form \(\A\) gives on \(\K_c\) the famous \emph{Minkowski inequality} (or \emph{Alexandrov--Fenchel inequality} in higher dimensions): 
\begin{equation}\label{eq minkowski}
\a(K_1,K_2)^2\geq \a(K_1)\a(K_2)~,
\end{equation}
with equality if and only if \(K_1=\lambda K_2\) for \(\lambda>0\).

Also, \begin{equation}\label{eq:perimeter}\pi_0(\supp(K))=\frac{1}{2\pi}\p(K)~,
\end{equation} where \(\p(K)\) is the perimeter of \(K\). Moreover, 
\eqref{eq minkowski} with \(K_2=D\) is the isoperimetric inequality.
\end{remark}

If \(\K_c^\pi\) denotes the subset of elements of \(\K_c\) of area \(\pi\), any \(K\in \K_c^\pi\) satisfies \(\A(\supp(K)) = 1\) by formula \eqref{eq:area equality}, and \(\pi_0(\supp(K)) > 0\), since \(\supp(K)\) is a positive map. Therefore, the image of \(\K_c^\pi\) by \(\supp\) is a subset of \(\mathbb{H}^\infty\).  Let \(\supp^\pi\) be the restriction of \(\supp\) to \(\K_c^\pi\).

\begin{lemma}\label{lem:vitale}
The preimage by 
\[\supp^\pi\colon (\K_c^\pi,d_{\operatorname{Haus}}) \to (\mathbb{H}^\infty,d_{\H^\infty}) \]
of a bounded set is bounded.
\end{lemma}
\begin{proof}
Let \(X\) be a set in \(\mathbb{H}^\infty\) bounded for \(d_{\H^\infty}\). By Lemma~\ref{lem bounded 1}, it is bounded for the \(H^1\)-norm. Then it is bounded for the \(L^2\)-norm. So there exists a \(C\) such that for any \(h\in X\cap \supp^\pi(\K_c^\pi)\), the \(L^2\)-distance between \(\o\) and \(h\) is less than \(C\). By a theorem of Vitale  \cite[Corollary 2]{vitale}, there exists a positive dimensional constant \(\alpha\) such that
\begin{align}\label{eq:Vitale}
\alpha \frac{\|h-\o\|_{L^\infty}^{3}}{1+\|h-\o\|_{L^\infty}} \leq  \|h-\o\|_{L^2}^2~. 
\end{align}
Hence, there is a constant \(C'\) such that
\[ \frac{\|h-\o\|_{L^\infty}^{3}}{1+\|h-\o\|_{L^\infty}} \leq  C'\]
that implies that \(\|h-\o\|_{L^\infty}\) is uniformly bounded over \(X\cap \supp^\pi(\K_c^\pi)\).
\end{proof}

\begin{lemma}\label{lem:supp conti inj}
The map \(\supp^\pi:(\K_c^\pi,d_{\operatorname{Haus}})\to (\H^\infty,d_{\H^\infty})\) is a continuous injective map.
\end{lemma}
\begin{proof}
Injectivity comes from the fact that convex bodies are uniquely determined by their support function. 
Let us check the continuity. From Lemma~\ref{lem:supp cont H1},  \((h_i)_i\)
converge to \(h\) in \(L^2\) and  \((h'_i)_i\)
converge to \(h'\) in  \(L^2\), so it is immediate from the definition of \(\A\) that \(\A(h_i-h)\to 0\),
that gives  \(d_{\H^\infty}(h_i,h)\to 0\) by Remark~\ref{rem cv dH A}. 
\end{proof}

\begin{lemma}\label{supp proper}
The continuous map \(\supp^\pi:(\K_c^\pi,d_{\operatorname{Haus}})\to (\H^\infty,d_{\H^\infty})\) is proper, \emph{i.e.} the preimage of a compact set is compact.
\end{lemma}
\begin{proof}
Let \(X\subset (\H^\infty,d_{\H^\infty})\) be a compact subset. In particular, it is closed and bounded as we are in a metric space. 
Its preimage by \(\supp^\pi\) is closed (Lemma~\ref{lem:supp conti inj}) and bounded (Lemma~\ref{lem:vitale}), hence compact (Proposition~\ref{prop:blaschke}).
\end{proof}

\begin{proposition}\label{rem: supp homeo}
    The map \(\supp^\pi:(\K_c^\pi,d_{\operatorname{Haus}})\to (\H^\infty,d_{\H^\infty})\) is a homeomorphism onto its image.
\end{proposition}
\begin{proof}
This simply follows from the preceding results, as a continuous proper injective mapping from a locally compact space into a metric space is a homeomorphism onto its image.
\end{proof}

\begin{remark}\label{remark midpoint}
From the finite-dimensional case, 
the image of a geodesic between two points in \(\H^\infty\) is the intersection between  \(\H^\infty\) and the plane spanned by the lines in \(H^1_{\operatorname{even}}\) defined by the two points. From \eqref{eq:addition supp}, it follows that the image of the hyperbolic geodesic between \(\supp(K_0)\) and \(\supp(K_1)\) is the image of 
\[t\mapsto p(t)\coloneqq\frac{\supp(K_t)}{\A(K_t)^{1/2}}\]
where \(K_t=(1-t)K_0+tK_1\) for \(0\leq t\leq 1\).

It is a general fact that the midpoint \(h\) in  \(\H^\infty\) between two points \(h_1,h_2\in \H^\infty\) is the  projection onto \(\H^\infty\) (along the vector line of the midpoint of the segment) of the midpoint of the segment in \(H^1_{\operatorname{even}}\), \emph{i.e.} \(h=(h_1+h_2)/2\A((h_1+h_2)/2))^{1/2}\). Indeed, 
\[
d_{\H^\infty}(h_1,h)=\cosh^{-1}\frac{\A(h_1,h_2)+1}{\A(h_1+h_2)^{1/2}}=d_{\H^\infty}(h_2,h)~.
\]

In turn, if \(K_1,K_2\in \K_c^\pi\), the midpoint of their images by \(\supp\) in \(\H^\infty\) differ from \(K_1+K_2\) by a dilation, or equivalently, they are homothetic, \emph{i.e.} \(C\cdot\supp(K_1+K_2)\) for some \(C>0\).
\end{remark}


\subsection{Segments and ellipses}

In the following, we denote by \(\CH\) the closed convex hull, for the convexity which is defined according to the context. If we want to insist on which topology is being used, we use a subscript: \(\CH_{\operatorname{Haus}}\) (resp.   \(\CH_{H^1}\)) if one considers \(d_{\operatorname{Haus}}\) (resp.  \(\|\cdot \|_{H^1}\)), and so on. Let us note that the Hausdorff distance is naturally used over non-empty compact sets of the plane, and in this space, the subspace of convex bodies is closed \cite[Theorem 1.8.6]{schneider}.

The set of \emph{symmetric zonoids} 
is the closure (for the Hausdorff topology) of the convex hull of the set $\mathrm{S}$ of symmetric segments of positive length introduced in \eqref{eq:intro S}. 
By construction, a symmetric zonoid is centrally symmetric with respect to the origin.  The following famous result says that the converse holds.
This is false in higher dimensions, see for example \cite[Corollary 3.5.7]{schneider}. We provide a simple proof for convenience. 

\begin{lemma}\label{lem Z=C}
Each symmetric plane convex body is a symmetric zonoid. In other words,   \(\overline{\K}_c=\CH_{\operatorname{Haus}}(\mathrm{S})\). Moreover,  \(\overline{\K}_c=\CH_{\operatorname{Haus}}(\mathcal{E}_c)\) where \(\mathcal{E}_c\) is the space of symmetric ellipses.
\end{lemma}  
\begin{proof}
That is obvious for segments and the origin.
Any convex body of positive area  can be approximated by 
convex polygons in $\K_c$ for the Hausdorff distance, so it suffices to prove that any symmetric convex polygon  is a symmetric zonoid. Indeed, we have that any such polygon $P$ is a sum of symmetric segments, see Figure~\ref{fig:poly segment}. 

Any symmetric segment can be obtained as the limit of symmetric ellipses, hence \(\mathrm{S}\subset \CH_{\operatorname{Haus}}(\mathcal{E}_c)\), and then
\[\overline{\K}_c\subset \CH_{\operatorname{Haus}}(\mathrm{S})\subset \CH_{\operatorname{Haus}}(\mathcal{E}_c) \subset \overline{\K}_c\]
where the last inclusion is the closedeness of \(\overline{\K}_c\).
\end{proof}

\begin{figure}[h!]
\begin{center}
\psfrag{P}{$P$}
\psfrag{P1}{$P'$}
\psfrag{S}{$S$}
\includegraphics[width=0.5\linewidth]{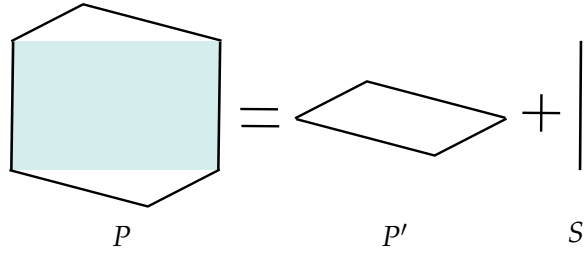}\caption{To the proof of Lemma~\ref{lem Z=C}. Sides of a  symmetric polygon \(P\) can be grouped into parallel pairs with equal lengths. For such a pair,  \(P\)  is the sum of a  symmetric convex polygon \(P'\) and a  symmetric segment \(S\), parallel to the elements of a pair of parallel sides.}\label{fig:poly segment}
\end{center}
\end{figure}

By the properties of additivity \eqref{eq:addition supp} and continuity (Lemma \ref{lem:supp cont H1}), and the fact that the image of \(\overline{\K}_c\) is closed (Lemma~\ref{lem:K barre ferme}) we obtain the following:

\begin{proposition}\label{prop: CH1}
The sets \(\supp(\overline{\K}_c)\), \(\CH_{H^1}(\supp(\mathrm{S})) \), and \(\CH_{H^1}(\supp(\mathcal{E}_c))\) coincide. \qeda
\end{proposition}
 

The following corollary of Proposition~\ref{prop: CH1} allows us to approximate elements in \(H^1_{\operatorname{even}}\) using support functions.

\begin{corollary}\label{cor: densite}
The differences of support functions form a dense vector subspace in \(H^1_{\operatorname{even}}\) (with respect to \(H^1\)-topology). In particular, the linear combinations of support functions of ellipses span a dense vector subspace of \(H^1_{\operatorname{even}}\).
\end{corollary}
\begin{proof}
We will first show that  differences of support functions generate a dense vector subspace. Note that any \(\pi\)-periodic \( C^\infty\) function over the circle can be written as the difference of two support functions. Indeed, a \(\pi\)-periodic function \(h\) on the circle is a support function if and only if its degree one homogeneous extension is convex, {\it i.e.} its Hessian is positive semi-definite, that is equivalent to \(h''+h\geq 0\). In turn, for any smooth function \(h\), there is a constant \(c\) such that \(h+c\o\) will achieve this condition. As smooth functions are dense in Sobolev space, the first assertion follows. The second assertion follows from this one together with Proposition~\ref{prop: CH1}.
\end{proof}


\section{An exotic representation via convex bodies}\label{sec : exotic rep}

In this section, we will see that the hyperbolic embedding of convex bodies induces an irreducible representation \(\psl\to\Isom(\H^\infty)\) and deduce from the geometric properties of the embedding some more descriptions about the representation with elementary arguments.

\subsection{The \(\psl\) action}\label{sec:the psl action}

The special linear group  \(\sl\) acts continuously onto  \(\R^2\), in turn it acts continuously onto \(\overline{\K}_c\) by the action
\[\sl \times  \overline{\K}_c \to \overline{\K}_c ~ , ~ (A,K)\mapsto AK\coloneqq\{Ak\mid k \in K\}.\]
As the convex bodies are  symmetric, we actually have a \(\psl\) action onto  \(\overline{\K}_c \), namely $[A]K\coloneqq AK$ for any representative $A\in\sl$. Abusing notation, we will often denote \([A]\) by \(A\).

There is a representation \(\rho\) of  
\(\sl\) onto the linear mappings of \(H^1_{\operatorname{even}}\) defined by, for \(A\in \sl\)
and \(h\in H^1_{\operatorname{even}}\) seen as a function on the circle \(\SS^1\subset \R^2\), 
\begin{equation}\label{eq:action supp fct}\rho(A)(h)(x)=|A^\intercal x| h(A^\intercal\cdot x)\end{equation}
where \(\cdot\) is the following action of \(\sl\) on the unit circle: 
\begin{align}\label{eq: PSL action on circle}
\forall x\in \SS^1\subset \R^2, M\cdot x\coloneqq \frac{Mx}{|Mx|}~ .
\end{align}
Note that the map \(\rho(A)(h)\) is even for \(h\in H^1_{\operatorname{even}}\), thus is actually well-defined for \(A\in \psl\). As the map is essentially the composition of a multiplication by a bounded map and a right composition by a diffeomorphism, its image stays in \(H^1_{\operatorname{even}}\). 

For \(K\in \overline{\K}_c  \), from the definition of \(\supp(K)\) \eqref{eq:def supp} one has,
\begin{equation}\label{eq:equivarience supp}\supp(AK)=\rho(A)(\supp(K))~.\end{equation}
\begin{remark}\label{remark:orbit ellipse} In the following, we will often use the fact that  
$\Ec\subset \mathcal{K}_c^\pi$,  the collection of symmetric ellipses in $\mathbb{R}^2$  of area $\pi$, is the orbit of \(D\) for the action of \(\psl\). In turn, support functions of ellipses in \(\H^\infty\) are the \(\rho(\psl)\)-orbit of \(\o\).
 \end{remark}

\begin{lemma}\label{lem: invariant bilinear form} For any \(A\in \psl\),
\(\A\) is \(\rho(A)\)-invariant on \(H^1_{\operatorname{even}}\).
\end{lemma}
\begin{proof}
\(A\) preserves the area of symmetric plane convex bodies, hence by \eqref{eq:area equality} and \eqref{eq:equivarience supp}, \(\rho(A)\) preserves 
\(\A\) for the support functions of convex bodies. The sum of two support functions is a support function, so with
\[2\A(\supp(K_1),\supp(K_2))=\A(\supp(K_1)+\supp(K_2))-\A(\supp(K_1))-\A(\supp(K_2))\]
the bilinear form \(\A(\cdot,\cdot)\) is preserved by \(\rho(A)\) over the support functions. Hence as
\[\A(\supp(K_1)-\supp(K_2))
=\A(\supp(K_1)) + \A(\supp(K_2))-2\A(\supp(K_1),\supp(K_2)),\]
 \(\rho(A)\) preserves \(\A\) 
for the difference of support functions. Then Corollary~\ref{cor: densite} implies that \(\A\) is \(\rho(A)\)-invariant on the entire \(H^1_{\operatorname{even}}\).
\end{proof}

\begin{lemma}
The representation \(\rho\) is continuous with respect to the strong operator topology (SOT): for every \(h\in H^1_{\operatorname{even}}\), we have \(\rho(A_n) h\to h\) whenever \(A_n\to \Id\in\psl\). 
\end{lemma}
\begin{proof}
The action is clearly continuous for \(C^\infty\) functions, that are dense in \(H^1_{\operatorname{even}}\). 
\end{proof}





Restricting to \(\H^\infty\subset H^1_{\operatorname{even}}\), we obtain a representation that we still denote by \(\rho\),
 \[\rho:\psl\to \operatorname{Isom}(\H^\infty)~.\]

\begin{remark}\label{faithful}
The representation \(\rho\) is faithful, since for each non-trivial \(A\in\psl\), one can find a convex body \(K \in \K_c^\pi\) such that \(AK\neq K\).\end{remark}


We now show that this representation \(\rho\colon \psl\to \Isom(\H^\infty)\) is irreducible. 
For the following lemma, compare with \cite[Lemma 3.9]{MP}.
\begin{lemma}\label{lemma:* fixed point}
The only fixed point in \(\H^\infty\) for the action of \(\rho(\so)\) is \(\o\).
\end{lemma}
\begin{proof}
From the action \eqref{eq:action supp fct}, one sees that $\rho(\so)$ acts on $H^1_{\operatorname{even}}$ by precomposing the functions by a rotation of the circle $\SS^1$. Since $\rho(\so)$ acts transitively on $\SS^1$, the only fixed points of this action are the constant functions in the line $L_0 \subset H^1_{\operatorname{even}}$. As $L_0\cap \H^\infty$ is reduced to $\{\o\}$, $\o \in \H^\infty$ is the unique fixed point for the $\rho(\so)$ action on $\H^\infty$. 
\end{proof}
Let \(\mathcal{H},\langle\cdot,\cdot\rangle)\) be a Hilbert space endowed with an indefinite, non-degenerate and continuous bilinear form \(B\) of finite index. For every continuous linear operator \(T\) of \(\mathcal{H}\), we denote by \(T^\dagger\) its \(B\)-adjoint: \(T^\dagger\) is the unique continuous linear operator such that \(B(Tx,y) = B(x,T^\dagger y)\) for all \(x,y \in \mathcal{H}\), \emph{viz.} \(T^\dagger= J^{-1} T^*J\), with \(J=J^*\) such that \(B(x, y) = \langle Jx, y \rangle\). If \(\mathcal{H}\) can be decomposed as a \(B\)-orthogonal direct sum \(\mathcal{H} = V \oplus V^{\perp_B}\), then the \(B\)-orthogonal projection \(P_V\) onto \(V\) is a well-defined bounded linear operator on \(\mathcal{H}\). Moreover, \(P_V\) is \(B\)-self-adjoint, \(P_V^\dagger = P_V\).

Recall also that for a Hilbert space \(\mathcal{H}\) and a collection of bounded linear operators \(M\subset \mathcal{B}(\mathcal{H})\), the {\it (bounded) commutants} of \(M\), denoted by \(M'\), is the collection of the bounded linear operators on \(\mathcal{H}\) that commute with all operators \(T\in M\), {\it i.e.}
\[M'\coloneqq\{S\in \mathcal{B}(\mathcal{H})\mid TS=ST,~\forall T\in M\}~.\]
Then we have the following:
\begin{lemma}\label{lem: commutant}
Let \(\mathcal{H}\) be a Hilbert space with indefinite, non-degenerate and continuous bilinear form \(B\) of finite index and \(M\subset \mathcal{B}(\mathcal{H})\) be an \(B\)-self-adjoint collection of bounded linear operators. If \(M'=\{v\mapsto rv \mid r\in \R\}\), then \(M\) does not preserve any proper non-degenerate closed subspace of \(\mathcal{H}\).
\end{lemma}
\begin{proof}
Suppose that \(M\) has a proper closed invariant subspace \(V\subset \mathcal{H}\) which is non-degenerate for the bilinear form \(B\). Since \(B\) has finite index, one can decompose \(\mathcal{H}\) as the \(B\)-orthogonal direct sum \(\mathcal{H} = V \oplus V^{\perp_B}\), see for example \cite[\S 1.7, Theorem 7.16]{azizov-iokhvidov}. Then consider the orthogonal projection \(P_V\colon \mathcal{H}\to V\). It is not difficult to verify that \(P_V\in M'\) by checking directly the commutativity because every element in \(M\) has to preserve the decomposition \(\mathcal{H} = V \oplus V^{\perp_B}\). But the only possibility that such an orthogonal projection \(P\) is of the form \(v\mapsto rv\) for some \(r\in\R\) is when \(r=0\), {\it i.e.} \(V = \{0\}\).
\end{proof}

In our setting, \(H^1_{\operatorname{even}}\) is a Hilbert space and the bilinear form \(\A\) has index \(1\). Note also that, by Lemma \ref{lem: invariant bilinear form}, \(\rho(A)^\dagger = \rho(A^{-1})\) for all \(A \in \psl\).
With Lemma~\ref{lem: commutant}, we are then able to deduce the irreducibility of the representation \(\rho\).
\begin{proposition}\label{prop : irreducible}
The representation \(\rho\colon \psl\to \Isom(\H^\infty)\) is irreducible.
\end{proposition}
\begin{proof}
We first remark that \(\A\) is not degenerate on any closed subspace in \(H^1_{\operatorname{even}}\) spanned by a hyperbolic subspace of \(\H^\infty\). Due to our construction of \(\H^\infty\), to show the irreducibility of \(\rho\), it is sufficient to prove that the \(\psl\)-action on \(H^1_{\operatorname{even}}\) defined by \eqref{eq:action supp fct} doesn't preserve any proper non-degenerate closed subspace in \(H^1_{\operatorname{even}}\) nor any isotropic line. By Lemma \ref{lemma:* fixed point}, the action of \(\rho(\psl)\) on \(\H^\infty\) cannot fix a point on the boundary \(\partial_\infty\H^\infty\), therefore \(\rho(\psl)\) doesn't preserve any isotropic line in \(H^1_{\operatorname{even}}\). By Lemma~\ref{lem: commutant}, it remains to show that its commutants are multiplications by constants. Indeed, let \(P\) be a commutant of \(\rho(\psl)\). Then we have
\[\rho(A)P\o=P\rho(A)\o=P\o\]
for any \(A\in \mathrm{SO}_2(\R)<\psl\). This implies that \(P\o\) is also a constant function \(c\o\). More generally, for any \(A\in \psl\), we have after \eqref{eq:equivarience supp}
\[P\supp(AD)=\rho(A)P\o=c\rho(A)\o=c\supp(AD)~.\]
Since \(P\) is a linear map, we have \(Ph=ch\) for every \(h\) that is a linear combination of support functions of ellipses, see Remark~\ref{remark:orbit ellipse}. Since \(P\) is bounded and is thus continuous, now applying Corollary~\ref{cor: densite}, we can conclude that \(P\colon h\mapsto ch\) for every \(h\in H^1_{\operatorname{even}}\).
\end{proof}

\begin{remark}\label{rep}
The group obtained by composing elements of \(\psl\) by     \(S=\begin{pmatrix} 0 & 1 \\ 1 & 0\end{pmatrix}\) and the identity, which is isomorphic to \(\Isom(\H^2)\), also acts
aver  \(\overline{\K}_c\) and over \(H^1_{\operatorname{even}}\) by
\[\rho(S)(h)(x)=h(S(x))~.\]
From \eqref{eq: def A}, it is immediate that \(\rho(S)\) is an isometry for \(\A\). So we could have worked with a representation \(\rho:\Isom(\H^2) \to \Isom(\H^\infty)\) instead of \(\rho:\psl \to \Isom(\H^\infty)\). 

That won't change our statements  ---in the former case, we obtain the Banach--Mazur compactum instead of the oriented Banach--Mazur compactum  in \ref{T6} of Theorem~\ref{prop:somme ellipse} and  in Theorem~\ref{prop BM}.
\end{remark}

\begin{remark}
In \cite{MP},  continuous representations of \(\psl\to\Isom(\H^\infty)\) are also constructed by building non-totally geodesic \(\psl\)-equivariant embeddings of \(\H^2\) into \(\H^\infty\). But the ambient space \(\H^\infty\) constructed in \cite{MP} can admit a \(\psl\)-invariant proper subspace. In order to obtain irreducible representations, it is necessary to pass to a proper subspace of \(\H^\infty\), which remains still infinite-dimensional. In contrast, for the construction via convex bodies, Proposition~\ref{prop : irreducible} shows that there is no need to pass to a subspace.
\end{remark}

\subsection{Exotic embedding of the hyperbolic plane}

Recall that $\Ec\subset \mathcal{K}_c^\pi$ is the collection of symmetric ellipses in $\mathbb{R}^2$  of area $\pi$.  This space is homogeneous for $\sl$ since every ellipse of \(\Ec\) can be realised as an image of the unit disk $D\in\Ec$ under a linear transformation of $\sl$. First note that from \eqref{eq:action supp fct}
and \eqref{eq:equivarience supp}, the support function of an ellipse is given by 
\begin{equation}\label{eq:sup fct ellipse}\supp(AD)(x)=|A^\intercal x|~.\end{equation}

Note also that for each ellipse $E\in \Ec$, its stabiliser $\Stab(E)<\sl$ is conjugate to $\so$ in $\sl$. It turns out that this space can be identified with the real hyperbolic plane $\mathbb{H}^2$. More precisely, if we consider the Poincar{\'e} upper half-plane model for $\mathbb{H}^2$, then we have a well-defined map

 
    \begin{align}\label{eq:E}
\begin{split}
\operatorname{E} \colon\H^2&\longrightarrow\E_c^\pi\subset \K_c^\pi\\
A(i) &\longmapsto AD
\end{split}
\end{align}
where \(A(i)\) is the action by homography on \(i\in \H^2\), \emph{i.e.} if \(A=\begin{pmatrix} a & b \\ c & d \end{pmatrix}\), \(A(i)=\frac{ai+b}{ci+d}\).

\begin{remark}
Let us denote \(T_s\coloneqq\operatorname{diag}(e^{2s},e^{-2s})\). It is easy to see that 
\[d_{\H^2}(i,T_s(i))=s\]
and also that for any \(A\in \psl\),  there is \(s\in \R\) and \(R,R'\in \so\) such that \(A=R'T_sR\). 
\end{remark}



\begin{lemma}\label{lem E plongement} The map \(\operatorname{E}\) defined by \eqref{eq:E} is a homeomorphism onto its image \(\Ec\subset(\overline{\K}_c,d_{\operatorname{Haus}})\).
\end{lemma}
\begin{proof} The map
\(\operatorname{E}\) is clearly a bijection between \(\H^2\) and \(\Ec\), and is clearly equivariant for the action of   \(\psl\), which acts continuously on both \(\Ec\) and \(\mathbb{H}^2\). Hence
it is a continuous injective map. We will check that it is a proper map, that reduces to show that the preimage of a bounded set is bounded. 

Let \(A\in \sl\) such that \(d_{\operatorname{Haus}}(AD,D)<C\).
 By Lemma~\ref{topo unif f s} and \eqref{eq:action supp fct}, it is easy to see that rotations are isometries for the Hausdorff distance, so we may only consider the case when \(A=T_s=\operatorname{diag}(e^{s/2}, e^{-s/2})\).
By our assumption, using Lemma~\ref{topo unif f s} again and
\eqref{eq:sup fct ellipse}, we have
\[C> d_{\operatorname{Haus}}(T_sD,D)=\sup_\theta \left| \left|T_s^\intercal \binom{\cos \theta}{\sin\theta }\right|-1  \right|=|e^{s/2}-1|.\]
Hence, in our case, \(1\leq e^{s/2}<1+C\), and 
\(d_{\H^2}(T_s(i),i)=s\) is clearly bounded.
\end{proof}





\begin{proposition}\label{prop : smooth}
The map \[\iota\coloneqq\supp^\pi\big|_{\Ec}\circ \operatorname{E}:\H^2\to \H^\infty\] is a \(C^\infty\)-embedding.
\end{proposition}
\begin{proof}
By Proposition~\ref{rem: supp homeo} and Lemma~\ref{lem E plongement}, \(\iota\) is a topological embedding. As \(\rho(\psl)\) acts by isometries onto \(\iota(\H^2)\), it suffices to prove that it is \(C^\infty\) in a neighbourhood of \(i\in\H^2\) and of maximal rank at \(i\), {\it i.e.} the image of the derivative \(\mathrm{d}_i\iota \colon T_i\H^2\to T_{\iota(i)}\H^\infty\) is of maximal dimension \(2\).

Let \(s>0\) and \(A_{t,s}=\begin{pmatrix} s  &  st \\ 0 & 1/s
\end{pmatrix}\).
We will parametrise a neighbourhood of \(i\) in the Poincaré upper half-plane by a neighbourhood \(U\) of \((0,1)\in \R^2\)  using
\[(t,s)\mapsto (A_{t,s}(i))=s^2i+s^2t~. \]
Abusing notation, we denote \(\iota:U\to \H^\infty\) the composition of \(\iota\) with the above map.

As
\[\iota(t,s)=\iota(A_{t,s}(i))=\rho(A_{t,s})\iota(i)=\rho(A_{t,s})\o\]
by \eqref{eq:action supp fct},
\[\iota(t,s)(\theta)=\left|A_{t,s}^\intercal \binom{\cos\theta}{\sin\theta}\right|~.\]
Note that if \((t,s)\) is sufficiently close to \((0,1)\), then \((t,s,\theta)\mapsto \iota(t,s)(\theta)\) is uniformly bounded from below by a positive constant \(c\). Indeed,
let  \(0<a<1\) and  \(0<c<a\). Then for any \(s,1/s\in(a,\frac{1}{a})\), \(|t| < a^2-c^2\), and for all \(\theta\in[0,2\pi]\), 
\begin{align*}
\iota^2(t,s)(\theta)= &\left|A_{t,s}^\intercal \binom{\cos\theta}{\sin\theta}\right|^2=(s^2+s^2t^2)\cos^2\theta + \frac{1}{s^2}\sin^2\theta + 2t\cos \theta \sin \theta \\
\geq &s^2 \cos^2 \theta + \frac{1}{s^2}\sin^2 \theta +  2t\cos \theta \sin \theta \geq a^2  - |t|>c^2~.
\end{align*}
Similarly, we can show that it is also uniformly bounded from above by a constant \(C>0\).

We will consider \(\iota\) as a map   from \(U\) to 
\((H^1_{\operatorname{even}},\|\cdot  \|_{H^1})\), that will give the wanted result by Lemma~\ref{lem:meme topo}. Let \(X\) be the matrix \(\begin{pmatrix} 0 & 0 \\ 1 & 0 \end{pmatrix}\), so that \(A_{t+h,s}^\intercal=A_{t,s}^\intercal+shX\). 
The functions
\begin{align*}
f_h(\theta)\coloneqq&\frac{\iota(t+h,s)(\theta)-\iota(t,s)(\theta)}{h}=\frac{\iota^2(t+h,s)(\theta)-
\iota^2(t,s)(\theta)}{h(\iota(t+h,s)(\theta)+
\iota(t,s)(\theta))} \\
=&\frac{\langle sX\binom{\cos\theta}{\sin\theta},(2A_{t,s}^\intercal +shX)\binom{\cos\theta}{\sin\theta} \rangle}{\iota(t+h,s)(\theta)+
\iota(t,s)(\theta)}
\end{align*}
pointwise converge when \(h\to 0\) to 
\[\frac{\partial \iota}{\partial t}(t,s)(\theta)=\frac{\langle \binom{0}{s\cos},A_{t,s}^\intercal \binom{\cos\theta}{\sin\theta} \rangle }{\iota(t,s)(\theta)}~.\]
For \(|h|\) sufficiently small, the functions \(f_h\) are bounded by a constant function that does not depend on \(h\). Hence by the Lebesgue dominated convergence theorem, \(f_h\) converges to \(\frac{\partial \iota}{\partial t}(t,s)\) in \(L^2(\SS^1)\) as \(h\to 0\). 
It is also clear that as a function in \(\theta\), the function \(f_h'\) pointwise converges to \(\frac{\partial \iota}{\partial t}(t,s)'\), and is uniformly bounded by a constant function that does not depend on \(h\). Then by the Lebesgue dominated convergence theorem, \(f_h'\) converges in \(L^2(\SS^1)\) to \(\frac{\partial \iota}{\partial t}(t,s)'\). As consequence, \(f_h\) converges to \(\frac{\partial \iota}{\partial t}(t,s)\) in \(H^1_{\operatorname{even}}\), that is then the first partial derivative of \(\iota\) at \((t,s)\).




In the same way, a direct computation shows that
the limit when \(h\to 0\) of 
\[g_h(\theta)=\frac{\iota(t,s+h)(\theta)-\iota(t,s)(\theta)}{h}\]
exists in \(H^1_{\operatorname{even}}\) 
and is equal to
\[\frac{\partial \iota}{\partial s}(t,s)(\theta)= \frac{\left\langle \begin{pmatrix} 1 & 0 \\ t & -\frac{1}{s^2} \end{pmatrix}\binom{\cos\theta}{\sin\theta},A_{t,s}^\intercal \binom{\cos\theta}{\sin\theta} \right\rangle}{\iota(t,s)(\theta)}~.\]
The same argument gives that those two partial derivatives are continuous in \(H^1_{\operatorname{even}}\) around \((0,1)\in\H^2\). Actually, it is not difficult to see that all derivatives (with respect to \(s,t\) or \(\theta\)) of those functions are uniformly bounded by constants, hence applying recursively the arguments above, it follows that \(\iota\) is a \(C^\infty\)-map.


At \((0,1)\in\H^2\),
\[\frac{\partial \iota}{\partial t}(0,1)(\theta)=\cos(\theta) \sin(\theta),\quad \frac{\partial \iota}{\partial s}(0,1)(\theta)=2\cos^2(\theta) - 1\]
and it is easy to verify that these two vectors are independent in \(H^1_{\operatorname{even}}\), hence \(\iota\) has maximal rank at \((0,1)\).
\end{proof}

We now compute the extrinsic distance between two points of \(\iota(\H^2)\).
\begin{lemma}\label{lem_dh}
For any $A\in \sl$, we have
\begin{equation}
    \label{eq:extrinsic distance}
    d_{\mathbb{H}^\infty}(\iota(A(i)),\iota(i))=\cosh^{-1}\Big(\frac{1}{2\pi}\int_0^{2\pi}|A^\intercal(\cos\theta,\sin\theta)^\intercal\big|~\mathrm{d}\theta\Big).
\end{equation}
\end{lemma}
\begin{proof}
Let $A\in \sl$. The distance  induced by \(\A\) onto \(\mathbb{H}^\infty\) is given by 
\[d_{\H^\infty}(h,k)=\cosh^{-1}\A(h,k)\]
hence  we have
\begin{align*}
    d_{\mathbb{H}^\infty}(\iota(A(i)),\iota(i))&=\cosh^{-1}\Big(\A(\supp(AD),\supp(D))\Big)\\
    & =\cosh^{-1}\Big(\A(\supp(AD),\o)\Big) \\
    &=\cosh^{-1}\Big(\frac{1}{2\pi}\int_0^{2\pi} \supp(AD)\Big),
\end{align*}
and the result follows by \eqref{eq:sup fct ellipse}.
\end{proof}


Let \(g_\iota\) be the pull-back over the upper half-plane by \(\iota\) of the metric induced by \(g_{\mathbb{H}^\infty}\) over \(\iota(\Ec)\).

\begin{proposition}\label{prop: curvature}
We have  \(g_\iota = \frac{3}{8}g_{\H^2}\), in particular the metric \(g_\iota\)   has constant negative curvature \(-\frac 8 3\).
\end{proposition}
\begin{proof}
The group
\(\psl\) acts isometrically and transitively on the tangent vectors of both metrics, so it suffices to prove the assertion at one point and at one tangent vector at that point.  For $s>0$, on the one hand, for \(d_{\H^2}\) the path \(c(s)=\begin{pmatrix} e^{s/2} & 0 \\ 0 & e^{-s/2}\end{pmatrix}(i)\) is a geodesic, hence
\[g_{\H^2}(c'(0),c'(0))=1~,\]
and on the other hand
\begin{align*}
d_{\H^\infty}\big(\iota (c(s)),\o\big)&\stackrel{ \eqref{eq:extrinsic distance}}{=}\cosh^{-1}\left(\frac{1}{2\pi}\int_0^{2\pi}\sqrt{e^{s}\cos(\theta)^2+e^{-s}\sin(\theta)^2}~\mathrm{d}\theta\right)\\
&=\cosh^{-1}\left(1+\frac{3 s^2}{16}+o(s^2)\right)\\
&=\sqrt{2}\sqrt{\frac{3}{16}s^2}+O((s^2)^{3/2})=\sqrt{\frac{3}{8}}s+ O(s^3)~.
\end{align*}
Therefore,
\[\sqrt{g_{\iota}(c'(0),c'(0))}=
\sqrt{g_{\H^\infty}(\iota_*(c'(0)),\iota_*(c'(0)))}=\lim_{s\searrow 0}\frac{d_{\mathbb{H}^\infty}\big(\iota(c(t)),\o\big)}{s}=\lim_{s\searrow 0}\left(\sqrt{\frac{3}{8}}+O(s^2)\right)\]
which proves the desired result.
\end{proof}

Since the curvature of \(\iota(\H^2)\) is lower than the ambient curvature \(-1\), \(\iota(\H^2)\) is not totally geodesic.

The above descriptions lead to the study of the shape of the surface $\iota(\H^2)\subset \H^\infty$.

\begin{proposition}\label{prop: minimal surface}
The surface \(\iota(\H^2)\) is a minimal surface.
\end{proposition}
\begin{proof}
The argument comes from \cite{DP}, we give it in full details for convenience. 
Let \((e_1,e_2)\) be an orthonormal basis of \(T_{\o}\iota(\H^2)\). The surface \(\iota(\H^2)\) is minimal if
\[\eta\coloneqq-\tilde{D}^N_{\tilde{e}_1}\tilde{e}_1-\tilde{D}^N_{\tilde{e}_2}\tilde{e}_2=0\]
for any extension \(\tilde{e_i}\) of \(e_i\) (\(i=1,2\)) to \(\mathbb{H}^\infty\), where \(\tilde{D}^N_{\tilde{e}_i}\tilde{e}_i\) is the normal part of
the \(\H^\infty\)  Levi-Civita connection 
 \(\tilde{D}_{\tilde{e}_i}\tilde{e}_i\), see \cite[Remark 5.21]{GHL}. Note that as a trace, \(\eta\) does not depend on the choice of the orthonormal basis.

Suppose that \(\eta\) is non zero. Then it is transverse to  $\iota(\H^2)$. Let \(f\in \rho(\so)\). The point \(\o\) is fixed by \(f\), and \(f\) acts isometrically on \(T_{\o}\iota(\H^2)\). It is also 
an isometry of \(T_{\o}\H^\infty\), and preserves the orthogonal splitting given by  \(T_{\o}\iota(\H^2)\) and its orthogonal part, as well as the connection. In turn,  the vector \(\eta\) is invariant for \(f\).
Hence, the geodesic passing through \(\o\) and directed by  \(\eta\) is invariant for \(f\), and as \(f\) is an isometry, it acts by translation on the geodesic. But \(f\) fixes a point on this geodesic, this implies that it fixes the geodesic pointwise. This contradicts the fact that \(\o\) is the only fixed point of \(f\) ({\it cf.} Lemma~\ref{lemma:* fixed point}). So \(\eta\) must be zero.
\end{proof}

\begin{proposition}\label{prop: quasi iso}
There exists a constant \(C\) such that for all \( x,y \in \H^2\),
$$\left|d_{\mathbb{H}^\infty}(\iota(x),\iota(y))-\frac{1}{2}d_{\mathbb{H}^2}(x,y)\right|\leq C~.$$
\end{proposition}
\begin{proof}
By transitivity of the action of \(\psl\) over \(\H^2\), it suffices to prove the result for \(x=i\) and \(y=T_s(i)\), with \(T_s=\operatorname{diag}(e^{s/2},e^{-s/2})\).

Note that $\left|T_s^\intercal(\cos\theta,\sin\theta)^\intercal\right|=\sqrt{e^{s}\cos^2\theta+e^{-s}\sin^2\theta}$. We can estimate that $$\left|T_s^\intercal(\cos\theta,\sin\theta)^\intercal\right|\leq \sqrt{e^{s}(\cos^2\theta+\sin^2\theta)}=e^{s/2}$$
and that
$$\left|T_s^\intercal(\cos\theta,\sin\theta)^\intercal\right|\geq \sqrt{e^{s}\cos^2\theta}=e^{s/2}|\cos\theta|~.$$
Hence $\cosh^{-1}(2e^{s/2}/\pi)\leq d_{\mathbb{H}^\infty}(\iota(T_s(i)),\iota(i))\leq \cosh^{-1}(e^{s/2})$ after Lemma~\ref{lem_dh}.
On the one hand \(d_{\mathbb{H}^2}(T_s (i),i)=s\),  and on the other hand
 \(\ln(x)<\cosh^{-1}(x)=\ln(x+\sqrt{x^2-1})<\ln(2x)\), therefore 
\[
d_{\mathbb{H}^\infty}(\iota(T_s(i)),\iota(i))\leq\cosh^{-1}(e^{s/2})<\ln(2e^{s/2})=\frac{s}{2}+\ln(2),
\]
and
\[d_{\mathbb{H}^\infty}(\iota(T_s(i)),\iota(i))\geq\cosh^{-1}(2e^{s/2}/\pi)>\ln(2e^{s/2}/\pi)=\frac{s}{2}+\ln(2/\pi),\]
which is sufficient to conclude the desired result by taking \(C=\ln(2)\).
\end{proof}

The fact that the Minkowski sum of two collections of pairwise non-homothetic ellipses do not coincide seems to be well-known, but we are unable to find a reference, so we will provide here a simple proof, inspired by Will Sawin \cite{MO}. Note that the same arguments remain valid for ellipsoids in higher dimensions as well.

Let \(E\coloneqq AD\) for some \(A\in \sl\). First, from \eqref{eq:sup fct ellipse}, we can see that the support function of an ellipse is the square root of a quadratic form. In particular, considering the one-homogeneous extension of  the support function to \(\R^2\), we can write for any \((x,y)\in \R^2\)
\[\supp(E)(x,y)=\sqrt{Q(x,y)}~,\]
where \(Q(x,y)\) is a quadratic form on \(\R^2\). The support function \(\sqrt{Q(x,y)}\) extends to a multi-branched holomorphic function on \(\mathbb{C}^2\setminus \Lambda_1\cup \Lambda_2\), where the \(\Lambda_i\)'s are the two complex lines of singularities of \(\sqrt{Q(x,y)}\). To be more precise, if
\[A=\begin{pmatrix}
    a & b\\ c& d
\end{pmatrix}\in \sl,\]
then \(Q(x,y)=\big((a-ib)x+(c-id)y\big)\big((a+ib)x+(c+id)y\big)\), and hence, if we set \(z_1\coloneqq a+ib\), \(z_2=c+id\) and \(\tau\colon A\mapsto -z_1/z_2\), then the two lines of the singularities are \(\{y=\tau(A)x\}\) and \(\{y=\overline{\tau(A)}x\}\) (note that because \(A\in \sl\), \(z_i\not=0\) for \(i=1,2\)).

\begin{lemma}\label{lem: ellipse injective map}
Let \(\tau\colon \sl\to \mathbb{C}\) be as above. For \(A,A'\in \sl\), one has \(\tau(A)=\tau(A')\) if and only if \(A'=A R\) for some \(R\in \so\). Moreover, the map \(\widetilde{\tau}\colon \Ec\to \mathbb{C}\) defined by \(AD\mapsto \tau(A)\) is a well-defined injective map.
\end{lemma}
\begin{proof}
The second result of the lemma is implied by the first result. Now, let us prove the first result.  For the ``only if'' part, let \(z_1,z_2\) be as above and \(z_1',z_2'\) be the corresponding quantities for \(A'\). By our assumption, we have
\[-\frac{z_1}{z_2}=\tau(A)=\tau(A')=-\frac{z_1'}{z_2'},\]
which implies \(z_1'/z_1=z_2'/z_2=:s\in \mathbb{C}\setminus\{0\}\). Moreover, \(A,A'\in \sl\) indicates that 
\[1=\mathrm{Im}(\overline{z_1}'z_2')=\mathrm{Im}(|s|^2\overline{z_1}z_2)=|s|^2,\]
hence \(s=e^{i\theta}\) for some \(\theta\in [0,2\pi)\). As \(z_i'=sz_i\) for \(i=1,2\), we can compute
\[A'=\begin{pmatrix}
    a\cos\theta-b\sin\theta & a\sin\theta+b\cos\theta \\  c\cos\theta-d\sin\theta & d\cos\theta+c\sin\theta
\end{pmatrix}=A\begin{pmatrix}
    \cos\theta & \sin\theta\\ -\sin\theta & \cos\theta
\end{pmatrix}\eqqcolon AR~,\]
for \(R\in \so\).

From the computation above, the ``if'' part is trivial. Indeed, if \(A'=AR\), then \(z'_i=e^{i\theta}z_i\) for \(i=1,2\), hence \(\tau(A)=\tau(A')\).
\end{proof}

Now, we are able to conclude the following:
\begin{proposition}\label{prop: sum minkowski}
    Let \(E_1,\dots, E_n\in\E_c^\pi\) be pairwise non-homothetic ellipses. Then for any \(1\leq m< n\) and any positive constants \((c_k)_{1\leq k\leq n}\), the Minkowski sums \(\sum_{k=1}^m c_k E_k\) and \(\sum_{k=m+1}^n c_k E_k\) are different.
\end{proposition}
\begin{proof}
Consider the one-homogeneous multi-branched extension of \(\supp(E_k)\) to \(\mathbb{C}^2\), namely \(\sqrt{Q_k(x,y)}\) described as above, with singularities \(\Lambda_1^k\cup \Lambda^k_2\), where \(\Lambda^k_1\) and \(\Lambda^k_2\) are two complex lines. If \(E_k= A_k D\), then we may write \(\Lambda^k_1 = \{(x,\tau(A_k)x), x\in \mathbb{C}\} \subset \mathbb{C}^2\) and \(\Lambda^k_2 = \{(x,\overline{\tau(A_k)}x), x\in \mathbb{C}\} \subset \mathbb{C}^2\). If \(E_i = A_i D\) and \(E_j = A_j D\) are non-homothetic ellipses, then \(\tau(A_i) \neq \tau(A_j)\) and \(\tau(A_i) \neq \overline{\tau(A_j)}\). Therefore, \(\Lambda^i_r\neq \Lambda^j_s\) for \(r,s\in\{1,2\}\) by Lemma~\ref{lem: ellipse injective map}. However, the function
\[\varphi_1\coloneqq \sum_{k=1}^m c_k \sqrt{Q_k}\]
has singularities \(\bigcup_{k=1}^m(\Lambda^k_1\cup\Lambda^k_2)\), whereas
\[\varphi_2\coloneqq\sum_{k=m+1}^n c_k \sqrt{Q_k}\]
has singularities \(\bigcup_{k=m+1}^n(\Lambda^k_1\cup\Lambda^k_2)\). So the two functions \(\varphi_1,\varphi_2\) do not have the same set of singularities and thus cannot be the same function. But they are, respectively, the one-homogeneous extensions of \(\supp\left(\sum_{k=1}^m c_k E_k\right)\) and \(\supp\left(\sum_{k=m+1}^n c_k E_k\right)\). If the Minkowski sums \(\sum_{k=1}^m c_k E_k\) and \(\sum_{k=m+1}^n c_k E_k\) were the same, so would \(\varphi_1\) and \(\varphi_2\), which is not true.
\end{proof}

A consequence of the above proposition is the following result on the dimension of hyperbolic space spanned by points on the minimal surface \(\iota(\H^2)\subset\H^\infty\):

\begin{corollary}\label{cor: k dimension}
    Any \(n+1\) distinct points in \(\iota(\H^2)\subset\H^\infty\) generate a hyperbolic subspace of dimension \(n\) in \(\H^\infty\).
\end{corollary}
\begin{proof}
To conclude the desired result, it suffices to show that the associated \(n+1\) support functions \((\supp(E_k))_{1\leq k\leq n+1}\) are linearly independent in \(H^1_{\operatorname{even}}\). Suppose for contradiction that it is not the case, then we will have
\[\supp\Big(\sum_{k=1}^m c_k E_k\Big)=\sum_{k=1}^m c_k \supp(E_k)=\sum_{k=m+1}^{n+1} c_k \supp(E_k)=\supp\Big(\sum_{k=m+1}^{n+1} c_k E_k\Big)\,,\]
for some \(c_k\geq 0\)  and \(k=1,\dots,n\), with some non-zero \(c_k\). However, this will imply \(\sum_{k=1}^m c_k E_k=\sum_{k=m+1}^{n+1} c_k E_k\), which is absurd after Proposition~\ref{prop: sum minkowski}.
\end{proof}

\subsection{The closed convex hull}\label{sec Ck}

In this subsection, we derive several properties of \(C_\K\coloneqq\supp^\pi(\K_c^\pi)\).

Here, we recall that a subset \(C\subset (X,d)\) in a geodesic metric space is {\it convex} if for any \(x,y\in C\), all minimizing geodesics between \(x\) and \(y\) are contained in \(C\). 

\begin{proposition}\label{prop:Ck ferme}
\(C_\K\) 
is a closed convex set in \((\H^\infty,d_{\H^\infty})\), and a proper metric space for the induced distance.
\end{proposition}

\begin{proof}
It follows from \eqref{eq:addition supp} and \eqref{eq:lambda supp}   that \(\supp(\overline{\K}_c)\) is a convex cone in \(H^1_{\operatorname{even}}\), and hence \(\supp^\pi(\K_c^\pi)\) is a convex subset of \(\H^\infty\).
The closeness is given by Lemma~\ref{lem:K barre ferme} and Lemma~\ref{lem:meme topo}.

Let \(X\) be a bounded closed set of \(\supp^\pi(\overline{\K}_c)\). By Lemma~\ref{lem:supp conti inj} its preimage is closed, and bounded by Lemma~\ref{lem:vitale}, hence compact by 
Proposition~\ref{prop:blaschke}. So \(X\) is compact, again by the continuity of \(\supp^\pi\).
\end{proof}

As segments have zero area, \(\supp(\mathrm{S})\) is a subset of the isotropic cone of \(\A\). In the following we will identify it as a subset of the boundary at infinity of \(\H^\infty\). In the rest of this subsection, we will write \(\CH_{\H^\infty}\) the closed convex hull in \(\H^\infty\). 

Let \(\Gamma\subset \partial_\infty \H^\infty\) be the subset defined by the image of the segments \(\supp(\mathrm{S})\). Let \(\CH_{\H^\infty}(\Gamma)\) be  the intersection of all the closed convex subsets of \(\H^\infty\) whose boundary at infinity contains \(\CH_{\H^\infty}(\Gamma)\). With these notations, Proposition~\ref{prop: CH1} gives the following:

\begin{corollary}\label{cor: CH}
We have \(C_\K=\CH_{\H^\infty}(\Gamma)\), \(\partial_\infty C_K=\Gamma\) and \(C_\K=\CH_{\H^\infty}(\iota(\H^2))\). \qeda
\end{corollary}


Now we can deduce the following result. Its proof is taken from Lemma 4.1 in \cite{MP}. 
\begin{lemma}\label{lem:ck minimal}
The subset \(C_\K\) is the minimal (with respect to inclusion) closed convex invariant subspace of \(\H^\infty\) for \(\rho(\psl)\)
\end{lemma}
\begin{proof}
If \(C'\) is a  \(\rho(\psl)\)-invariant closed convex set, it must contain \(\o\) by Lemma~\ref{lemma:* fixed point}. Hence it must contain the orbit of \(\o\) that is \(\iota(\H^2)\), and by Corollary~\ref{cor: CH}, it contains \(C_\K\).
\end{proof}

Note that \(\H^\infty\) is a \(\operatorname{CAT}(-1)\) space. For any closed convex subset in \(\H^\infty\), there exists a unique point with the shortest \(d_{\H^\infty}\) distance to \(\o\in \H^\infty\). By \eqref{eq:perimeter}, the distance
\[d_{\H^\infty}(\o,h)=\cosh^{-1}\pi_0(h)~,\]
and in particular, if \(h=\supp(K)\in \supp(\K_c^\pi)\), then
\[d_{\H^\infty}(\o,h)=\cosh^{-1}\left(\frac{1}{2\pi}\p(K) \right).\]
As \(\cosh^{-1}\) increases monotonically, having the shortest \(d_{\H^\infty}\) distance for \(h\) in a closed convex subset of \(\supp(\K_c^\pi)\) is equivalent to being the optimal solution to the isoperimetric problem in this convex subset.

From the above discussions as well as Remark~\ref{remark midpoint}, we can see that the bi-infinite geodesic line in \(C_\K\subset\H^\infty\) connecting \(\nu,\omega\in \Gamma\) consists of the support functions of parallelograms of area \(\pi\)  homothetic to \(av+bw\) for some \(a,b>0\), where \(\supp(v)=\nu\) and \(\supp(w)=\omega\). If, in particular, we wish to find the nearest point projection of \(\o\) to this bi-infinite geodesic line, it is equivalent to find the optimal solution to the isoperimetric problem among these parallelograms.

Indeed, we can deduce the following result that will be used:
\begin{lemma}\label{lem: AM-GM}
    Let \(\nu,\omega\in \Gamma\) be distinct. Let \(h\) be the nearest point projection of \(\o\) to the bi-infinite geodesic line connecting \(\nu,\omega\in \Gamma\), with \(v,w\in \SS\) unit symmetric segments satisfying  \(\supp(v)=\nu\) and \(\supp(w)=\omega\). Then \(h=\supp(P)\) for \(P\in \K_c^\pi\) homothetic to \(v+w\), \emph{i.e.} a rhombus.
\end{lemma}
\begin{proof}
    Note that \(\p(av+bw)=2(a+b)\) and \(\a(av+bw)=ab\det(v,w)=\pi\). By the AM-GM inequality, as \(ab=\pi/\det(v,w)\) is a constant, the quantity \(a+b\) is minimised if and only if \(a=b\), and it is when \(d_{\H^\infty}(h,\o)\) is minimised.
\end{proof}

It is proved in Proposition 4.4 in \cite{MP} that 
\(\rho(\psl)\) is the isometry group of \(C_\K\) using representation theory. Here below, we will provide an alternative proof with more geometric flavour.
\begin{proposition}\label{isom Ck}
The isometry group \(\Isom(C_\K)\) is isomorphic to \(\psl\) and is exactly the subgroup in \(\Isom(\H^\infty)\) that preserves \(C_\K\). 
\end{proposition}
\begin{proof}
For the metric induced over \(C_\K\) by the hyperbolic metric, any element in \(\rho(\psl)\) induces an isometry on \(C_\K\). Moreover, \(\rho\) is faithful (Remark~\ref{faithful}), hence  
 we can view \(\psl\) as a subgroup of \(\Isom(C_\K)\). Now it suffices to show it is the entire group.

Given any \(g\in \Isom(C_\K)\), it sends bi-infinite geodesics to bi-infinite geodesics, {\it i.e.} it may be identified with a bijection among symmetric parallelograms (a sum of two segments) of area \(\pi\). More specifically, let \(P=\sqrt{\pi}[0,1]^2\) be the centred square in \(\K_c^\pi\). By Lemma~\ref{lem: AM-GM}, it is the nearest point projection of \(\o\) to a bi-infinite geodesic. Since \(g\) is an isometry, then \(gP\) will be nearest point projection of \(g\o\) to the \(g\)-image of that bi-infinite geodesic, which is again a parallelogram. Let \(v_1,v_2\) be the two independent vectors such that \(\sqrt{\pi}v_1\) and \(\sqrt{\pi}v_2\) are the vectors of two sides of the parallelogram \(gP\) and such that the symmetric unit segment parallel to \(v_1\) (resp. \(v_2\)) is the \(g\)-image of the vertical (resp. horizontal) symmetric unit vector, viewed as elements in \(\Gamma\). In particular, we have \(|\det(v_1,v_2)|=\a(gP)/\pi=1\). Now the isometry \(g\) can be associated to the unique element in \(\psl\) sending \(\big((0,1),(1,0)\big)\) to \((v{_1,v_2})\). Moreover, this gives a well-defined group homomorphism \(\Isom(C_\K)\to\psl\), of which the inverse is the inclusion \(\psl\hookrightarrow\Isom(C_\K)\).
\end{proof}


Classical properties of convex bodies immediately translate into the following result. The first part was proved in  Proposition 4.3 in \cite{MP}. Due to our characterisation of the convex hull, we can obtain a more precise description.

\begin{theorem}\label{prop BM}
The action of \(\psl\) on \(C_\K\) is proper and cocompact. Moreover, the quotient of \(C_\K\)  is homeomorphic to the 2-dimensional oriented Banach--Mazur compactum.
\end{theorem}

Indeed, it is classical that the action of the affine group on convex bodies (with non-empty interior) gives a Hausdorff  compact quotient (a \emph{compactum}). It is for example a step in the proof of Benzécri’s compactness theorem about convex bodies in the projective plane, see \cite{goldman}. It also appears in the study of the (2-dimensional) Banach--Mazur compactum. Recall that it is the set of isometry classes of 2-dimensional Banach spaces, endowed with a natural distance. It is a compact metric space. It is naturally homeomorphic to the quotient of \(C_\K\) (endowed with the topology given by the Hausdorff distance) by \(\operatorname{GL}_2(\R)\). Its topological properties are well studied, see e.g. \cite{BM3,BM2,BM1,BM4}. See \cite{Macbeath} for a different approach.

As we are considering equivalence classes of convex bodies with fixed area for the action of \(\sl\), our quotient is homeomorphic to the quotient of the space of convex bodies up to orientation preserving linear maps, that is a double cover of the Banach--Mazur compactum. It is called the \emph{oriented Banach--Mazur compactum} in \cite{BM3}. Those facts immediately yield Theorem~\ref{prop BM}.

In the references mentioned above, all the proofs are based on a continuous retract by deformation from the space of convex bodies to the space of ellipses (different ellipses can be successfully associated to a convex body, and their denominations are not uniform in the literature, see \cite[Section 10.12]{schneider}).

\subsection{Hausdorff dimension of the limit set}

    Since \(\Gamma\) is the boundary at infinity of \(\iota(\H^2)\subset\H^\infty\) and \(\rho(\psl)\) acts transitively on \(\iota(\H^2)\), it is clear that \(\Gamma\) is the limit set of \(\rho(\psl)\). In this subsection, we will compute explicitly its Hausdorff dimension with respect to the visual distance on \(\partial_\infty\H^\infty\).
    
    Let \(D_{\o}\) be the visual distance on \(\partial_\infty\H^\infty\) based at \(\o\). For our hyperbolic space, it has a very simple expression: if \(\zeta,\eta\in \partial_\infty\H^\infty\), \begin{equation}\label{eq:Do}D_{\o}(\zeta,\eta)=\sin\left(\frac{1}{2}\angle_{\o}(\zeta,\eta)\right)~\end{equation}
    where \(\angle_{\o}(\zeta,\eta)\in[0,\pi]\) is the angle for the hyperbolic metric 
    between the unique geodesic rays starting at \(\o\) toward \(\zeta\) and \(\eta\) respectively, see \cite[Section~3.5]{DSU}. 

Recall that \(\Gamma\subset \partial_\infty \H^\infty\) is made of isotropic vectors for \(\A\) that correspond to directions in the plane (\emph{i.e.}, each direction is a homothety class of symmetric segments).
We will denote by \(\angle(\zeta,\eta)\in [0,\pi]\) their angle in the plane.

\begin{proposition}\label{prop: visual dist}
For two elements \(\zeta,\eta\) in \(\Gamma\),
then
\[D_{\o}(\zeta,\eta)=\frac{\sqrt{\pi}}{2}\sqrt{\sin (\angle(\zeta,\eta))}~.\]
\end{proposition}
\begin{proof}
We write  \eqref{eq:Do}
under the form
\[D_{\o}(\zeta,\eta)=\sqrt{\frac{1-\cos(\angle_{\o}(\zeta,\eta))}{2}}\]
and we will compute the cosine of the angle.

 The points \(\zeta,\eta\in \partial_\infty\H^\infty\) correspond to two isotropic lines for \(\A\). Let \(v_1,v_2\) be two vectors colinear to those lines, such that \(\pi_0(v_i)= \A(\1,v_i)=1\). The vectors \(v_i-\1\) are orthogonal to \(L_0\) and may be identified with vectors of \(T_{\o}\H^\infty\), tangent to the geodesic rays of \(\H^ \infty\) from \(\1\) to \(\zeta\) and \(\eta\).
We have
\[\A(v_1-\1,v_2-\1)=\A(v_1,v_2)-1=\frac{1}{2}\A (v_1+v_2)-1\]
where the last equality holds because the vectors are isotropic. Similarly, \(\A(v_1-\1)=-1\), hence
\[\cos(\angle_{\o}(\zeta,\eta))=\frac{-\A(v_1-\1,v_2-\1)}{\sqrt{-\A(v_1-\1)}\sqrt{-\A(v_2-\1)}}=1-\frac{1}{2}\A (v_1+v_2)~,\]
and
\[D_{\o}(\zeta,\eta)=\frac{1}{2}\sqrt{\A(v_1+v_2)}~.\]

By Remark~\ref{eq:area equality}, \(\A (v_1+v_2)\) is \(\frac{1}{\pi}\) times the area of the plane parallelogram obtained by adding the segments directed by the direction determined by \(\zeta\) and \(\eta\). By \eqref{eq:perimeter}, those segments have perimeter \(2\pi\pi_0(v_i)=2\pi\), hence length \(\pi\). In turn, the area of the parallelogram is 
 \(\pi^2\sin(\angle(\zeta,\eta))\).
\end{proof}

We recall that for a compact metric space \((K,d)\), its {\it Hausdorff dimension} with respect to the distance \(d\) is defined as
\[\Hdim_d(K)\coloneqq\inf\{\delta>0 \mid H^\delta (K)=0\}~,\]
where
\[H^\delta (K)=\lim_{r\to 0+} \inf\big\{\sum_{i=1}^n \diam(U_i)^\delta \mid \bigcup_{i=1}^n U_i\supset K \ \operatorname{and}\ \diam(U_i)< r\big\}~.\]

\begin{remark}\label{remark Hdim}
Note that if \(0< t\leq 1\) and \(\lambda>0\), for a distance \(d\), \(\lambda d^t\) is also a distance. Moreover, one can easily check from the definition that
\[\Hdim_{\lambda d^t}(K)=\frac{1}{t}\Hdim_d(K)\]
for \(0<t\leq 1\) and \(\lambda>0\).
\end{remark}

From the explicit formula of the visual distance \(D_\o\) in Proposition~\ref{prop: visual dist}, we can further deduce the Hausdorff distance of \(\Gamma\):
\begin{corollary}\label{cor:haus}
The Hausdorff dimension of \(\Gamma\) for the visual distance \(D_\1\) on \(\partial_\infty\H^\infty\) is \(2\).
\end{corollary}
\begin{proof}
The distance \(\sin (\angle(\zeta,\eta))\) is at first order the distance \(\angle(\zeta,\eta)\) on the round circle, which makes the circle of Hausdorff dimension one. Then the desired result follows from Remark~\ref{remark Hdim}.
\end{proof}

More generally, if we consider other formulae of the visual distance in real hyperbolic space, then the result of the Hausdorff dimension of the exotic embedding of the boundaries \(\partial_\infty \H^n\to \partial_\infty \H^\infty\) will become a consequence of the quasi-isometry. Indeed, recall again that for \(n\geq 2\), the exotic representations of \(\Isom(\H^n)\to\Isom(\H^\infty)\) are completely classified in \cite{MP} and are parametrised by \(0<t<1\). Each exotic representation induces an orbit map \(f^n_t\colon \H^n\to \H^\infty\) verifying
\begin{equation}\label{eq: qi exotic}
\left|d_{\H^\infty}(f^n_t(x) ,f^n_t(y) )-td_{\H^n}(x,y)\right|\leq C
\end{equation}
for some constant \(C>0\) and any \(x,y\in \H^n\). Recall also that the \emph{Gromov product} on \(\H^n\times \H^n\) is defined as
\[\langle x,y\rangle_z\coloneqq \frac 1 2 \big(d_{\H^n}(x,z)+d_{\H^n}(y,z)-d_{\H^n}(x,y)\big)\]
and can extends to \(\partial_\infty \H^n\) by
\[\langle \xi,\eta\rangle_z\coloneqq \lim_{\substack{x\to \xi\\ y\to \eta}}\langle x,y\rangle_z~.\]
The same definition is valid for \(\H^\infty\) as well, see again \cite[Section~3.5]{DSU}.

\begin{proposition}\label{prop: hausdorff dim general}
    Let \(f^n_t\colon\H^n\to\H^\infty\) be the quasi-isometric embedding as above. Then \(f^n_t\) extends continuously to the boundary \(\partial f^n_t\colon \partial_\infty \H^n\hookrightarrow \partial_\infty\H^\infty\). Moreover, the Hausdorff dimension of \(\partial f^n_t(\partial_\infty\H^n)\) with respect to the visual distance on \(\partial_\infty\H^\infty\) is \(\frac{n-1}{t}\).
\end{proposition}
\begin{proof}
    Let \(D_o^n\) be the visual distance on \(\partial_\infty \H^n\) with respect to the unique fixed point \(o\in \H^n\). Let us assume for consistency \(f_t^n(o)=\o\in \H^\infty\) and let \(D_\1\) be the visual distance as above. Recall from \cite[Section~3.5]{DSU} that for \(\xi,\eta\in \partial_\infty\H^n\),
    \[D_o^n(\xi,\eta)=e^{-\langle \xi,\eta\rangle_o}\]
    where \(\langle \xi,\eta\rangle_o\) is the Gromov product in \(\H^n\). By \eqref{eq: qi exotic}, the embedding \(f_t^n\) is quasi-isometric, thus extends to the boundary. The continuity of \(\partial f_t^n\) can be seen from the comparison
    \[\big(D^n_o(\xi,\eta)\big)^t=e^{-t\langle \xi,\eta\rangle_o}\approx e^{-\langle \partial f_t^n(\xi),\partial f_t^n(\eta)\rangle_\1}=D_\o\big( \partial f_t^n(\xi),\partial f_t^n(\eta)\big)~,\]
    where \(A\approx B\) means that there exists a constant \(c>1\) such that \(c^{-1}B\leq A\leq c B\). In turn, the two distances \((D^n_o)^t\) and \(D_\o\) are bi-Lipschitz and thus \( \partial f_t^n\) is a continuous map. Moreover, as \( \partial f_t^n(\partial_\infty\H^n)\) is the continuous image of a compact set, it is itself compact and we can compute its Hausdorff dimension. Since \(\Hdim_{D_o^n}(\partial_\infty \H^n)=n-1\), the desired result soon follows from Remark~\ref{remark Hdim}.
\end{proof}

\section{Comparison with different constructions}\label{last sec}

In this section, we will characterise the hyperbolic embedding of convex bodies in more detail. We then make a connection with kernels of hyperbolic type,
and end this section by comparing our construction of the exotic representations with the others.

\subsection{Uniqueness of the embedding}\label{uniq embedding}

Given a continuous representation \(r\colon \psl \to \Isom(\H^\infty)\), if there exists a unique fixed point \(o\in \H^\infty\) of \(r(\so)\), then the orbit map of \(o\) is the only \(r\)-equivariant map. By Lemma~\ref{lemma:* fixed point}, we can deduce:

\begin{lemma}\label{lemma: uniqueness of map}
The map \(\supp\big|_{\Ec}\colon \Ec\to \mathbb{H}^\infty\) is the only \(\rho\)-equivariant map from \(\Ec\) into \(\H^\infty\). \qeda
\end{lemma}

We now prove that, up to a convexity condition,  \(\supp\big|_{\Ec}\) extends uniquely to
\(\K_c^\pi\).
Let us define
\[l(K_1,K_2,t)\coloneqq \pi \frac{(1-t) K_1+ tK_2}{\a((1-t) K_1+ tK_2)}\in \K_c^\pi\]
for any \(K_1,K_2\in \K_c^\pi\) and \(0\leq t\leq 1\). In Lemma~\ref{lemma: uniqueness of map}, the \(\rho\)-equivariant map  \(\supp\big|_{\Ec}\) can be uniquely extended from \(\Ec\) to the whole \(\K_c^\pi\), under the convexity condition that is given explicitly below:

\begin{proposition}\label{prop extension supp}
    Suppose that \(f\colon\K_c^{\pi}\to\H^\infty\) is a continuous \(\rho\)-equivariant map such that \(l(K_1,K_2,[0,1])\) is sent to the geodesic segment between \(f(K_1)\) and \(f(K_2)\) for all \(K_1,K_2\in \K_c^\pi\), then \(f\) coincides with \(\supp\).
\end{proposition}

For proving Proposition~\ref{prop extension supp}, let us see the following elementary lemma:
\begin{lemma}\label{lemma: triangle}
    Let \(OAB\) be a triangle in an affine space and let \(m_{OA}, m_{OB}\) be the midpoints of the segments \([OA]\) and \([OB]\). For any convex combination \(P_t = (1-t)A+tB \in [A,B]\), where \(t\) is a dyadic number in \((0,1)\), one can construct the segments \([O,P_t]\) with only straight lines passing through the points \(O,A,B,m_{OA},m_{OB}\) or the intersection points of already constructed lines. Moreover, each of these points \(P_t\) is adherent to the set of all such intersection points.
\end{lemma}
\begin{proof}
    It suffices to construct the midpoint \(P_{1/2}\) of the segment \([A,B]\). This allows to decompose the triangle \(OAB\) into two subtriangles \(OAP_{1/2}\) and \(OP_{1/2}B\) and iteratively divide these triangles to construct every dyadic point on \([A,B]\). Since the segment \([O,P_{1/2}]\) is a median of the triangle \(OAB\), it passes through its barycenter which is the intersection point between the two segments \([A,m_{OB}]\) and \([B,m_{OA}]\), see Figure~\ref{fig:constructing dyadic points}. Moreover, \([m_{OA},m_{OB}]\) and \([A,B]\) being parallel, Thales's theorem ensures that the intersection point between \([O,P_{1/2}]\) and \([m_{OA},m_{OB}]\) is at equal distance from \(O\) and \(P_{1/2}\), allowing to iterate this procedure.

    For the density part of the statement, observe that iterating the construction of Pappus's theorem from the two triples of points \((m_{OA},m_{OP},m_{OB})\) and \((A,P_{1/2},B)\), where \(m_{OP}\) is the midpoint of \([O,P_{1/2}]\), produces sequences of intersection points converging to \(P_{1/2}\), see Figure~\ref{fig:approaching dyadic points}. 
\end{proof}

\begin{figure}[ht]
\captionsetup[subfigure]{font=footnotesize}
\centering
\subcaptionbox{Construction of \(P_{1/2}\)}[.5\textwidth]{%
\begin{tikzpicture}[scale=.04]
    \coordinate[label=above:\(O\)] (O) at (104.10514016,171.41030551);
    \coordinate[label=left:\(A\)] (A) at (38.41398803,75.21395906);
    \coordinate[label=right:\(B\)] (B) at (195.45940913,75.21395906);
    \coordinate[label=left:\(m_{OA}\)] (mOA) at (71.25956409,123.31215118);
    \coordinate[label=right:\(m_{OB}\)] (mOB) at (149.78227654,123.31211339);
    \coordinate[label=below:\(P_{1/2}\)] (P12) at (116.93669669,75.21395906);
    
    \draw (O)--(A);
    \draw (A)--(B);
    \draw (O)--(B);
    \draw (mOA)--(mOB);
    \draw[myCyan] (A)--(mOB);
    \draw[myCyan] (B)--(mOA);
    \draw[myRed] (O)--(P12);
\end{tikzpicture}}%
\subcaptionbox{Construction of \(P_{1/4}\)}[.5\textwidth]{\begin{tikzpicture}[scale=.04]
    \coordinate[label=above:\(O\)] (O) at (104.10514016,171.41030551);
    \coordinate[label=left:\(A\)] (A) at (38.41398803,75.21395906);
    \coordinate[label=right:\(B\)] (B) at (195.45940913,75.21395906);
    \coordinate[label=left:\(m_{OA}\)] (mOA) at (71.25956409,123.31215118);
    \coordinate[label=right:\(m_{OB}\)] (mOB) at (149.78227654,123.31211339);
    \coordinate[label=below:\(P_{1/2}\)] (P12) at (116.93669669,75.21395906);
    \coordinate[label=below:\(P_{1/4}\)] (P14) at (77.67534236,75.21395906);
    
    \draw (O)--(A);
    \draw (A)--(B);
    \draw (O)--(B);
    \draw (mOA)--(mOB);
    \draw (A)--(mOB);
    \draw (B)--(mOA);
    \draw (O)--(P12);
    \draw[myCyan] (P12)--(mOA);
    \draw[myCyan] (A)--(110.52091843,123.31215118);
    \draw[myRed] (104.10514016,171.41030551)--(77.67534236,75.21395906);
\end{tikzpicture}}
\caption{Construction of dyadic points on a segment.}
    \label{fig:constructing dyadic points}
\end{figure}
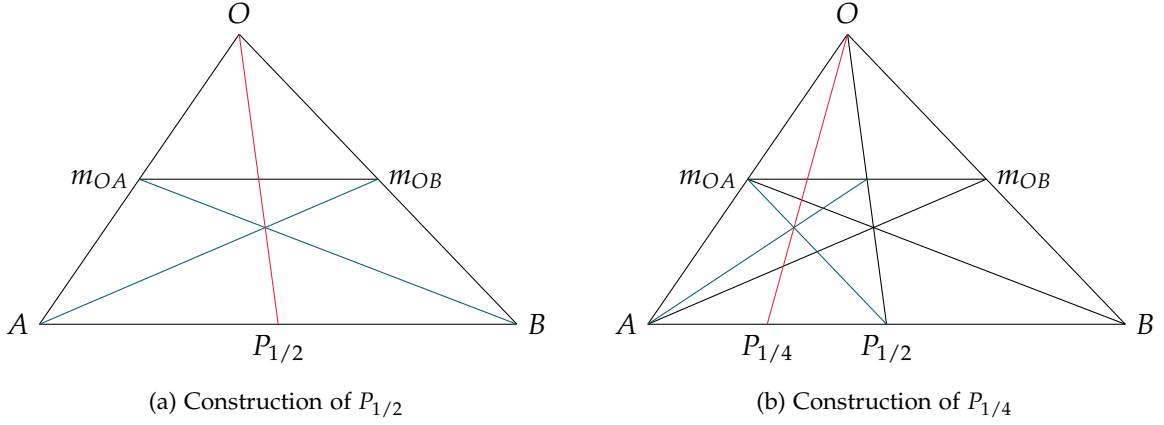

\begin{figure}[ht]
\captionsetup[subfigure]{font=footnotesize}
\centering
\subcaptionbox{First iteration}[.5\textwidth]{%
\begin{tikzpicture}[scale=.04]
    \coordinate[label=left:\(A\)] (A) at (38.41398803,75.21395906);
    \coordinate[label=right:\(B\)] (B) at (195.45940913,75.21395906);
    \coordinate[label=left:\(m_{OA}\)] (mOA) at (71.25956409,123.31215118);
    \coordinate[label=right:\(m_{OB}\)] (mOB) at (149.78227654,123.31211339);
    \coordinate[label=above:\(m_{OP}\)] (m) at (110.520920315,123.31211339);
    \coordinate[label=below:\(P_{1/2}\)] (P12) at (116.93669669,75.21395906);
    \coordinate[] (I1) at (86.4853,107.279);
    \coordinate[] (I2) at (112.66,107.279);
    \coordinate[] (I3) at (138.834,107.279);
    
    \filldraw[black] (mOA) circle (40pt);
    \filldraw[black] (mOB) circle (40pt);
    \filldraw[black] (m) circle (40pt);
    
    \draw (A)--(B);
    \draw (mOA)--(mOB);
    \draw[myCyan] (A)--(m);
    \draw[myCyan] (B)--(m);
    \draw[myCyan] (mOA)--(P12);
    \draw[myCyan] (mOB)--(P12);
    \draw[myCyan] (mOA)--(B);
    \draw[myCyan] (mOB)--(A);
    \draw[myRed] (I1)--(I3);

    \filldraw[myRed] (I1) circle (40pt);
    \filldraw[myRed] (I2) circle (40pt);
    \filldraw[myRed] (I3) circle (40pt);
\end{tikzpicture}}%
\subcaptionbox{Second iteration}[.5\textwidth]{\begin{tikzpicture}[scale=.04]
    \coordinate[label=left:\(A\)] (A) at (38.41398803,75.21395906);
    \coordinate[label=right:\(B\)] (B) at (195.45940913,75.21395906);
    \coordinate[label=left:\(m_{OA}\)] (mOA) at (71.25956409,123.31215118);
    \coordinate[label=right:\(m_{OB}\)] (mOB) at (149.78227654,123.31211339);
    \coordinate[label=above:\(m_{OP}\)] (m) at (110.520920315,123.31211339);
    \coordinate[label=below:\(P_{1/2}\)] (P12) at (116.93669669,75.21395906);
    \coordinate[] (I1) at (86.4853,107.279);
    \coordinate[] (I2) at (112.66,107.279);
    \coordinate[] (I3) at (138.834,107.279);
    \coordinate[] (I4) at (94.0983,99.2627);
    \coordinate[] (I5) at (113.729,99.2627);
    \coordinate[] (I6) at (133.36,99.2627);
    
    \filldraw[black] (mOA) circle (40pt);
    \filldraw[black] (mOB) circle (40pt);
    \filldraw[black] (m) circle (40pt);
    \filldraw[black] (I1) circle (40pt);
    \filldraw[black] (I2) circle (40pt);
    \filldraw[black] (I3) circle (40pt);
    
    \draw (A)--(B);
    \draw (mOA)--(mOB);
    \draw (mOA)--(I1);
    \draw (mOB)--(I3);
    \draw (I1)--(I3);
    \draw[myCyan] (I1)--(P12);
    \draw[myCyan] (I1)--(B);
    \draw[myCyan] (I2)--(A);
    \draw[myCyan] (I2)--(B);
    \draw[myCyan] (I3)--(A);
    \draw[myCyan] (I3)--(P12);
    \draw[myRed] (I4)--(I6);

    \filldraw[myRed] (I4) circle (40pt);
    \filldraw[myRed] (I5) circle (40pt);
    \filldraw[myRed] (I6) circle (40pt);
\end{tikzpicture}}
\caption{Pappus's theorem construction.}
    \label{fig:approaching dyadic points}
\end{figure}
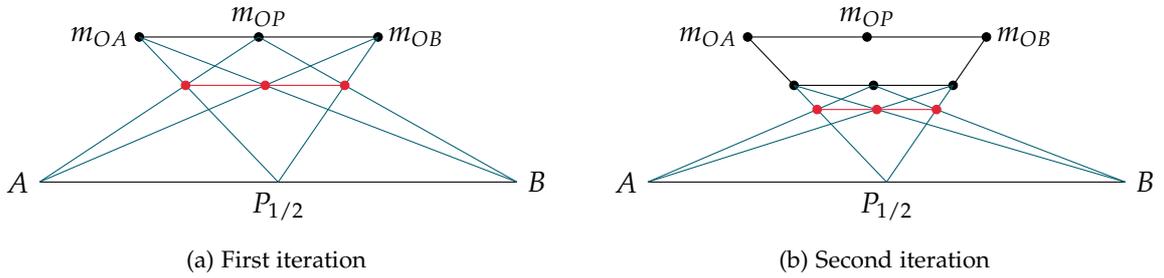

\begin{proof}[Proof of Proposition~\ref{prop extension supp}]
By Corollary~\ref{cor: CH}, the subspace in \(\K_c^{\pi}\) consisting of the convex combinations of finitely many ellipses is dense in \(\K_c^{\pi}\). By density, it suffices to show that the map \(f\) and \(\supp\) coincide on every convex combinations of finitely many ellipses with dyadic coefficients. We proceed by induction on the number of ellipses in the convex combination.
    
    By Lemma~\ref{lemma: uniqueness of map}, the map \(f\) coincides with \(\supp\) on \(\Ec\). Let \(E_1, E_2 \in \Ec\) be arbitrary and let \(P_t\coloneqq l(E_1,E_2,t)\) for \(t\in[0,1]\). By Remark~\ref{remark midpoint}, \(\supp(P_t)\) is the image of the geodesic segment between \(\supp(E_1)\) and \(\supp(E_2)\) in \(\mathbb{H}^\infty\) and \(\supp(P_{1/2})\) is the midpoint on this geodesic. By \(\rho\)-equivariance, \(f\) must also coincide with \(\supp\) on \(P_{1/2}\). Indeed, there exists \(A\in \psl\) such that \(AE_1 = E_2\) and \(AE_2 = E_1\). Therefore, \(\rho(A)\) exchanges \(f(E_1)\) and \(f(E_2)\) and fixes the point \(f(P_{1/2})\in\supp(P_t)\). Since the unique fixed point of \(\rho(A)\) on this geodesic segment is \(\supp(P_{1/2})\), we have \(f(P_{1/2}) = \supp(P_{1/2})\).

    By appealing to Lemma~\ref{lemma: triangle}, we can conclude that \(f(P_{t}) = \supp(P_{t})\) whenever \(t\in [0,1]\) is a dyadic number. Indeed, take another ellipse \(E_0\), then the previous discussion implies that \(f\) and \(\supp\) coincide at the midpoint on the segment of convex combinations between \(E_0\) and \(E_1\), as well as between \(E_0\) and \(E_2\), which correspond to the points \(m_{OA}\) and \(m_{OB}\) in Lemma~\ref{lemma: triangle}. By the convexity assumption on \(f\), if \(P_i\) and \(Q_i\) are convex bodies at which the images under \(f\) and \(\supp\) coincide for \(i=1,2\), and if \(R\) is the unique convex body that can written as a convex combination of both \(P_1, Q_1\) and \(P_2, Q_2\), then \(f\) and \(\supp\) also coincide at \(R\). Now Lemma~\ref{lemma: triangle} yields that for each dyadic number \(t\in[0,1]\), there exists a sequence \((R_n)_n\) converging to \(P_t\) in \(\K_c\) such that \(f(R_n) = \supp(R_n)\). Therefore, \(f(P_t) = \supp(P_t)\).

    In general, let \(\lambda_1,\dots, \lambda_n>0\) and let \(E_1,\dots,E_n\) be \(n\) convex bodies. We say that the Minkowski sum
    \[\lambda_1 E_1 + \lambda_2 E_2 + \cdots +\lambda_n E_n\]
    is a dyadic combination if there exists \(\lambda>0\) such that \(\lambda(\lambda_1+\lambda_2+\cdots+\lambda_n)=1\) and for each \(1\leq j\leq n\), we have \(\lambda\lambda_j=\alpha_k2^{-k}\) for some integers \(k>0\) and \(0\leq \alpha_k< 2^k\).
    
    For the inductive step, suppose that \(f\) coincides with \(\supp\) on all the dyadic combinations of \(n\) ellipses. Let \(E = \lambda_1 E_1 + \lambda_2 E_2 + \dots + \lambda_{n+1}E_{n+1}\) be such a dyadic combination of \(n+1\) ellipses. Suppose in addition that \(S\coloneqq\sum_{i=1}^{n+1}\lambda_i\). This combination can be decomposed in several ways, for example:
    
    \begin{align*}
        E &= \lambda_1 E_1 + (S-\lambda_1)\left(\frac{\lambda_2}{S-\lambda_1}E_2 + \dots+\frac{\lambda_{n+1}}{S-\lambda_1}E_{n+1}\right)\\
        &=  (S-\lambda_{n+1})\left(\frac{\lambda_1}{S-\lambda_{n+1}}E_1 + \dots+\frac{\lambda_{n}}{S-\lambda_{n+1}}E_{n}\right) + \lambda_{n+1} E_{n+1}.
    \end{align*}
    The points
    \[\hat{E}_1 \coloneqq \left(\frac{\lambda_2}{S-\lambda_1}E_2 + \dots+\frac{\lambda_{n+1}}{S-\lambda_1}E_{n+1}\right)\]
    and
    \[\hat{E}_{n+1} \coloneqq\left(\frac{\lambda_1}{S-\lambda_{n+1}}E_1 + \dots+\frac{\lambda_{n}}{S-\lambda_{n+1}}E_{n}\right)\]
    are then also dyadic combinations of \(n\) ellipses. Thus \(f(E)\) belongs to the intersection of the two geodesic segments
    \[[\supp(E_1),\supp(\hat{E}_1)]\]
    and
    \[[\supp(E_{n+1}),\supp(\hat{E}_{n+1})]~,\]
    which is reduced to the singleton \(\{\supp(E)\}\). Hence \(f(E) = \supp(E)\).
\end{proof}

\subsection{Kernels of hyperbolic type}\label{4.2}
The notion of {\it kernel of (real) hyperbolic type} was introduced in \cite{MP2} to embed a set into some real hyperbolic space, and it is used to construct exotic representations. However, we remind the reader that one should not confuse it with a kernel of complex hyperbolic type, both of which are defined in \cite{monod}. Without further mention, all kernels of hyperbolic type in the sequel will be of real hyperbolic type.

\begin{definition}[Kernel of hyperbolic type]
Let \(X\) be a non-empty set. A function \(\beta\colon X\times X\to \mathbb{R}\) is a \emph{kernel of (real) hyperbolic type} if it is symmetric, non-negative, and equal to \(1\) on the diagonal with
\begin{align}\label{hyp_form}
\sum_{i,j=1}^n c_i c_j \beta(x_i,x_j)\leq\left(\sum_{k=1}^n c_k\beta(x_k,x_0)\right)^2
\end{align}
for all \(n\geq 1\), any \(x_0,x_1,\dots,x_n\in X\) and any \(c_1,\dots,c_n\in \mathbb{R}\).
\end{definition}

This notion is closely related to another notion called {\it kernel of positive type}. Indeed, by rearranging the terms in \eqref{hyp_form}, we can see that \(\beta(\cdot,\cdot)\) is a kernel of hyperbolic type if and only if for one (and thus for any) \(z\in X\), the function
\begin{align}\label{eq: pd_form}
N_z(x,y)\coloneqq \beta(x,z)\beta(y,z)-\beta(x,y)
\end{align}
is a kernel of positive type, {\it i.e.} a symmetric function satisfying the inequality
\[\sum_{i,j=1}^n c_i c_j N(x_i,x_j)\geq 0\]
for all \(n\geq 1\), any \(x_1,\dots,x_n\in X\) and any \(c_1,\dots,c_n\in \mathbb{R}\).

Given a kernel of positive type \(N\colon X\times X\to \mathbb{R}\), the {\it GNS construction} (for Gelfand, Naimark, and Segal) will produce a Hilbert space \(\mathcal{H}\), unique up to isomorphism, and a map \(h\colon X\to \mathcal{H}\) such that \(\langle h(x),h(y)\rangle_\mathcal{H}=N(x,y)\) for any \(x,y\in X\). This space is a space of functions generated by \(\Phi_x(y)\coloneqq N(x,y)\). For a proof of this, we refer the reader to \cite[Theorem C.1.4]{Bekka2008kazhdan}.

Referring to the hyperboloid model \(\mathbb{H}^\infty\subset \mathbb{R}\oplus \mathcal{H}\), given a non-void set \(X\), if a function \(\beta\colon X\times X\to \mathbb{R}\) is a kernel of hyperbolic type and if \(N_{x_0}\colon X\times X\to \mathbb{R}\) is an associated kernel of positive type, then the map into hyperbolic space \(\phi\colon X\to \mathbb{H}^\infty\) will be given by
\[\phi(x)\coloneqq \big(\beta(x,x_0),h(x)\big)~,\]
where \(x_0\in X\) is a fixed base point.


For any quadratic form defining a \(\H^\infty\) as in \eqref{eq : lorentzian}, the restriction of \(B\) to \(\H^\infty\) is a kernel of hyperbolic type. Indeed, \eqref{hyp_form} is the reversed Cauchy--Schwarz inequality, see \eqref{eq:RCS}:
if \(X=\sum_{i=1}^n c_ix_i\) and \(Y=x_0\), then by bilinearity 
\eqref{hyp_form} is equivalent to
\[1=B(X,X)=B(X,X)B(Y,Y)\leq \left(B(X,Y)\right)^2~.\]

In particular, the restriction
of  our form \(\A\) to \(\H^\infty\)  is also a kernel of hyperbolic type. Since the definition of kernel of hyperbolic type obviously 
passes to the restriction to any subset, 
  \(\A\) is also a kernel of hyperbolic type on \(\supp^\pi(\K_c^\pi)\).  Interpreting this fact in term of mixed-area (see Remark~\ref{remark area}), we obtain the following result:
  \begin{proposition}\label{prop mink neg}
      Let \(K_0,K_1,\ldots ,K_n\) be symmetric plane convex bodies of positive area. Then for any \(c_1,\ldots,c_n\in \R\), then
      \[\a(K_0)\sum_{i,j=1}^n c_ic_j\a(K_i,K_j)\leq \left(\sum_{k=1}^n c_k \a(K_k,K_0) \right)^2~.\]
  \end{proposition}
We remark that the coefficients can be negative. If \(K=c_1K_1+\cdots+c_nK_n\) with all \(c_i\) being positive, then by properties of the mixed area, the inequality 
above writes
 \[\a(K_0)\a(K)\leq \left(\a(K,K_0) \right)^2~,\]
which is the Minkowski inequality \eqref{eq minkowski}.

For (not necessarily symmetric) convex bodies of higher dimensions (including infinite ones), the polarisation of the second intrinsic volume is also a kernel of hyperbolic type (it follows directly from the fact that it comes from a bilinear form of Lorentzian signature, \cite[Proposition 2.4]{DF}, or \cite[proof of Proposition 3.2]{long})  hence, the result of Proposition~\ref{prop mink neg} also holds in those cases.

\subsection{Relation with the Delzant--Monod--Py construction}\label{section: comparison}
We remark that on the same space, one can define different kernels which may yield embeddings of the space into \(\mathbb{H}^\infty\) with quasi-isometric images. But these quasi-isometries between different images may not necessarily reduce to an isometry.

For example, on the one hand, \cite[Corollary 8.3]{monod} states that \(e^{\frac{1}{2}d_{\mathbb{H}^2}}\) is also a kernel associated to a representation 
\[ \rho'\colon\psl\simeq\Isom^+(\mathbb{H}^2)\hookrightarrow\Isom(\mathbb{H}^\infty)~,\]
which yields a quasi-isometric $\rho'$-equivariant embedding $\psi\colon \mathbb{H}^2\hookrightarrow\mathbb{H}^\infty$ such that, by Lemma~\ref{lem: H2 dist}, 
\begin{align}\label{kern_2}
\cosh \Big(d_{\mathbb{H}^\infty}\big(\psi(A(i)),\psi(i)\big)\Big)= e^{\frac{1}{2}d_{\mathbb{H}^2}(A(i),i)}=\left(\frac{\|A\|^2_F}{2}+\sqrt{\frac{\|A\|^4_F}{4}-1}\right)^{\frac{1}{2}}
\end{align}
for any $A\in \sl$.

Recall that \(\|A\|_F\) is the {\it Frobenius norm} and is given by
$$\left\|\begin{pmatrix}a & b \\ c & d\end{pmatrix}\right\|_F=\sqrt{a^2+b^2+c^2+d^2}~.$$
The following distance formula is elementary:
\begin{lemma}\label{lem: H2 dist}
For any $A,B\in \sl$, $d_{\mathbb{H}^2}(A(i),B(i))=\cosh^{-1}\left(\|B^{-1}A\|_F^2/2\right)$. \qeda
\end{lemma}

 On the other hand, appealing to Lemma~\ref{lem_dh}, we have
\begin{align}\label{kern_3}
\cosh\Big(d_{\mathbb{H}^\infty}(\iota(A(i)),\iota(i))\Big)=\frac{1}{2\pi}\int_0^{2\pi}\big|A^\intercal(\cos\theta,\sin\theta)^\intercal\big|~\mathrm{d}\theta
\end{align}
for any $A\in \sl$. It is clear that the two quantities \eqref{kern_2} and \eqref{kern_3} are in general different when $A\notin \so$, since \eqref{kern_3} is an elliptic integral in terms of the entities of the matrix \(A\), while \eqref{kern_2} is an elementary function of these entities. Hence, the associated embeddings $\psi$ and $\iota$ of $\mathbb{H}^2$ into $\mathbb{H}^\infty$ are quasi-isometric but not isometric to each other. This is because the representation \(\rho'\) is not irreducible.

Suppose now we have two irreducible representations \(\psl\to \Isom(\mathbb{H}^\infty)\) and the \(\psl\)-orbits with respect to these two representations of the unique fixed point by \(\so\) in \(\mathbb{H}^\infty\) are quasi-isometric, then \cite[Theorem B]{MP} states that this quasi-isometry must reduce to an isometry. In \cite{DP} and \cite{MP}, using a different kernel of hyperbolic type, one can also produce an  exotic representation \(\rho_t\) of \(\psl\) in \(\Isom(\mathbb{H}^\infty)\), see the Introduction. 
From what was said just above, we know that \(f_{\frac 1 2 }(\H^2)\) is isometric to 
\(\iota(\H^2)\). In the rest of this section, we check this fact by a direct computation.

Let us first recall the construction from \cite{DP,MP}. Consider the action of \(\psl\simeq \Isom(\mathbb{H}^2)\) on the boundary \(\partial\mathbb{H}^2\simeq \SS^1\) of \(\mathbb{H}^2\) induced by its isometric action on the hyperbolic plane. As before, we parametrise the unit circle \(\SS^1\) by \((\cos(\theta),\sin(\theta))\) for \(\theta\in[0,2\pi)\). Then we can also define an action of \(\psl\) on \(L^2(\SS^1)\) by
\begin{align}\label{eq: action DP}
A.f\coloneqq \mathrm{Jac}(A^{-1})^{\frac{3}{2}}f\circ A^{-1}~,
\end{align}
for every \(A\in\psl\) and \(f\in L^2(\mathbb{R})\), where \(\mathrm{Jac}(A)\) is the Jacobian of \(A\colon \SS^1\to \SS^1\) with respect to the angular measure \(\mathrm{d}\theta\) on \(\SS^1\). We remark that although the two constructions eventually pass to a \(\psl\)-action on some space of functions on \(\SS^1\), the action given in \eqref{eq: action DP} is very different from the action defined in \eqref{eq:equivarience supp} using the support functions of convex bodies.

Now we embed \(\SS^1\) into \(\mathbb{C}\) by \((\cos(\theta),\sin(\theta))\mapsto e^{i\theta}\) for \(\theta\in[0,2\pi)\). For every non-negative integer \(n\geq 0\), let \(e_n\colon \SS^1\to \mathbb{C}\) be the function given by \(z\mapsto z^n\). Define the linear operator \(\mathcal{L}\colon L^2(\SS^1;\mathbb{C})\to L^2(\SS^1;\mathbb{C})\) by sending \(e_n\) to \(\lambda_n e_n\), where \(\lambda_0=1\) and
\[\lambda_n\coloneqq\prod_{k=0}^{n-1}\frac{2k-1}{2k+3}\]
for all \(n\geq 1\). Consider the bilinear form
\[\langle f,h\rangle_\mathcal{L}\coloneqq \frac{1}{2\pi}\int_0^{2\pi} f\overline{\mathcal{L}(h)}~\mathrm{d}\theta\]
defined for \(f,h\in L^2(\SS^1;\mathbb{C})\). We remark that the \(\psl\)-action on \(\SS^1\) defined above will also induce a group action on \(L^2(\SS^1;\mathbb{C})\), but the usual inner product is not preserved, whereas \(\langle\cdot,\cdot\rangle_\mathcal{L}\) is preserved. 

Restrict \(\mathcal{L}\) to \(L^2(\SS^1)\), the space of real-valued square integrable functions on \(\SS^1\). We still denote by \(\o\in L^2(\SS^1)\) the constant function on \(\SS^1\) of value \(1\). Let
\[\mathcal{X}\coloneqq \psl.\o\subset L^2(\SS^1)\]
be the orbit of the constant function \(\o\) under the action of \(\psl\) described above. \cite{DP,MP} showed that the bilinear form \(\langle\cdot,\cdot\rangle_\mathcal{L}\) is a \(\psl\)-invariant kernel of hyperbolic type, and yields an irreducible representation \(\rho_\mathcal{L}\colon \psl\to \Isom(\mathbb{H}^\infty)\). Moreover, as \(\o\) is the unique fixed point of \(\so<\psl\), the space \(\mathcal{X}\) can also be identified with the associated symmetric space \(\mathbb{H}^2\) and yields a quasi-isometric embedding \(\iota_\mathcal{L}\colon \mathcal{X}\simeq\mathbb{H}^2\hookrightarrow \mathbb{H}^\infty\). We remark that \(\rho_\mathcal{L}\) is exactly \(\rho_{\frac{1}{2}}\) in \cite{MP}.

Let \(A_t\coloneqq \mathrm{diag}(e^t, e^{-t})\in\psl\) for \(t\in\R\). In particular, we remark that \(A_t^{-1}=A_{-t}\).

\begin{lemma}\label{formula jacobian}
For \(t\geq 0\) and every \(\theta\in[0,2\pi)\), the Jacobian of the map \(A_{-t}\colon \SS^1\to\SS^1\) on the unit circle with respect to the angular measure \(\mathrm{d}\theta\) is given by
\[\mathrm{Jac}(A_{-t})(\theta)=\frac{1}{|\cosh(t)-\sinh(t)e^{i\theta}|^2}~.\]
\end{lemma}
\begin{proof}
The action of \(\psl\) on the unit circle \(\SS^1\subset\mathbb{C}\) factors through the homographic \(\mathrm{SU}(1,1)\)-action via the conjugation
\begin{align*}
\psl&\to \mathrm{SU}(1,1)\\
A&\mapsto CAC^{-1}
\end{align*}
where
\[C=\frac{1}{\sqrt{2}}\begin{pmatrix}
    1  &-i\\ 1 & i
\end{pmatrix}~.\]
Hence, we have \(\mathrm{Jac}(A_{-t})=\mathrm{Jac}(D_{-t})\), where
\[D_{-t}=\begin{pmatrix}
    \cosh(t) & -\sinh(t)\\
    -\sinh(t) & \cosh(t)
\end{pmatrix}\colon e^{i\theta}\mapsto \frac{\cosh(t)e^{i\theta}-\sinh(t)}{-\sinh(t)e^{i\theta}+\cosh(t)}~.\]
Now a simple computation of the derivatives yields the desired formula. 
\end{proof}

We notice that \(\mathcal{L}(\o)=\o\) and \(A_t.\o=\mathrm{Jac}(A_{-t})^{\frac{3}{2}}\o\). By the definition of hyperbolic distance, we have
\begin{align}\label{kern_1}
\cosh\big(d_\mathbb{\H^\infty}(\iota_\mathcal{L}(A_t.\o),\iota_\mathcal{L}(\o))\big)=\langle A_{t}.\o,\o\rangle_\mathcal{L}=\frac{1}{2\pi}\int^{2\pi}_0\frac{\mathrm{d}\theta}{|\cosh(t)-\sinh(t)e^{i\theta}|^3}~.
\end{align}
As commented earlier, since our representation \(\rho\colon\psl\to \Isom(\H^\infty)\) is conjugate to \(\rho_{\frac{1}{2}}\) by Corollary~\ref{cor:conjug}, the associated embeddings of \(\H^2\) via \(\iota_\mathcal{L}\) and \(\iota\) must be isometric, {\it i.e.} the formulae \eqref{kern_3} and \eqref{kern_1} are the same.

Before showing this, we will need an auxiliary lemma. First, we define
\[I(k)\coloneqq\int_0^{\frac{\pi}{2}}\frac{\mathrm{d}u}{\big(1-k^2\sin^2(u)\big)^{\frac{3}{2}}}~,\]
\[E(k)\coloneqq\int_0^{\frac{\pi}{2}}\sqrt{1-k^2\sin^2(u)}~\mathrm{d}u~,\]
and 
\[K(k)\coloneqq\int_0^{\frac{\pi}{2}}\frac{\mathrm{d}u}{\sqrt{1-k^2\sin^2(u)}}~.\]
We recall that \(K(k)\) is called the {\it complete elliptic integral of the first kind}, while \(E(k)\) is called the {\it complete elliptic integral of the second kind}.
\begin{lemma}
For any \(0< k<1\), we have
\begin{align}\label{eq: elliptic integral}
I(k)=\frac{E(k)}{1-k^2}~.
\end{align}
\end{lemma}
\begin{proof}
Recall a standard relation between the complete elliptic integrals and their derivatives (see for example \cite[\S 22.736]{analysis}):
\[K'=\frac{E}{k(1-k^2)}-\frac{K}{k}~.\]
Now we can check that \(I=kK'+K\). Indeed, we can compute
\[K'(k)=\int_0^{\frac{\pi}{2}}\frac{k\sin^2(u)}{\big(1-k^2\sin^2(u)\big)^{\frac{3}{2}}}~\mathrm{d}u~,\]
so 
\[kK'(k)+K(k)=\int_0^{\frac{\pi}{2}}\frac{k^2\sin^2(u)}{\big(1-k^2\sin^2(u)\big)^{\frac{3}{2}}}~\mathrm{d}u+\int_0^{\frac{\pi}{2}}\frac{1-k^2\sin^2(u)}{\big(1-k^2\sin^2(u)\big)^{\frac{3}{2}}}~\mathrm{d}u=I(k)~.\]
Combining the two relations, we get the desired result.
\end{proof}

The proof of the isometry between the two images relies on simplifying the distance formula into standard elliptic integrals of the second kind.
\begin{proposition}
    For \(t\geq 0\), we have
    \begin{align}\label{eq: kern_1 kern_3}
    \frac{1}{2\pi}\int^{2\pi}_0\frac{\mathrm{d}\theta}{|\cosh(t)-\sinh(t)e^{i\theta}|^3} = \frac{1}{2\pi}\int^{2\pi}_0\left|\begin{pmatrix}
        e^t & 0\\
        0 & e^{-t}
    \end{pmatrix}\begin{pmatrix}
        \cos{\theta}\\
        \sin{\theta}
    \end{pmatrix}\right|\mathrm{d}\theta~.
    \end{align}
In turn, the images \( \supp^\pi(\E^\pi_c)=\iota(\H^2)\) and \(\iota_\mathcal{L}(\mathcal{X})\) are isometric.
\end{proposition}

\begin{proof}
Since \(\iota_\mathcal{L}\) and \(\iota\) are both \(\psl\)-equivariant, by identifying \(\o\in \mathcal{X}\) and \(D\in\E^\pi_c\), up to translations and rotations by elements in \(\psl\), to show that the two images are isometric, it suffices to prove for any \(t\geq 0\),
\[d_\mathbb{\H^\infty}(\iota_\mathcal{L}(A_t.\o),\iota_\mathcal{L}(\o))=d_\mathbb{\H^\infty}(\iota(A_tD),\iota(D))~,\]
which further reduces to showing the equality \eqref{eq: kern_1 kern_3}.

Let us denote by \(I_1(t)\) the left-hand side of \eqref{eq: kern_1 kern_3} and by \(I_2(t)\) the right-hand side of \eqref{eq: kern_1 kern_3}. First, notice that
\begin{align}\label{eq: trigo norm}
|\cosh(t)+\sinh(t)e^{i\psi}|^2=\cosh(2t)+\sinh(2t)\cos(\psi)~.\
\end{align}
So by exploiting the trigonometric identities, we have
\begin{align*}
    I_1(t)&\stackrel{\psi=\theta-\pi}{=}\frac{1}{2\pi}\int^{\pi}_{-\pi}\frac{\mathrm{d}\psi}{|\cosh(t)+\sinh(t)e^{i\psi}|^3}\\
    &\stackrel{\eqref{eq: trigo norm}}{=}\frac{1}{2\pi}\int^{\pi}_{-\pi} \big(\cosh(2t)+\sinh(2t)\cos(\psi)\big)^{-\frac{3}{2}}~\mathrm{d}\psi\\
        &\stackrel{2u=\psi}{=}\frac{1}{\pi}\int^{\frac{\pi}{2}}_{-\frac{\pi}{2}} \big((\cosh(2t)+\sinh(2t))-2\sinh(2t)\sin^2(u)\big)^{-\frac{3}{2}}~\mathrm{d}u\\
        &=\frac{2}{\pi}\int_0^{\frac{\pi}{2}} \big(e^{2t}-(e^{2t}-e^{-2t})\sin^2(u)\big)^{-\frac{3}{2}}~\mathrm{d}u\\
        &=\frac{2e^{-3t}}{\pi}\int_0^\frac{\pi}{2}\big(1-(1-e^{-4t}))\sin^2(u)\big)^{-\frac{3}{2}}~\mathrm{d}u\\
        &=2e^{-3t}I(k)/\pi~,
\end{align*}
where we set \(k\coloneqq\sqrt{1-e^{-4t}}\in(0,1)\), {\it viz.} \(t>0\). Now using the relation \eqref{eq: elliptic integral}, on the one hand we can deduce that for all \(t>0\)
\[I_1(t)=2e^t E(k)/\pi~.\]
On the other hand, we observe that
\begin{align*}
    I_2(t)&=\frac{e^t}{2\pi}\int_0^{2\pi}\sqrt{1-k^2\sin(u)^2}~\mathrm{d}u=\frac{2e^t}{\pi}\int_0^{\frac{\pi}{2}}\sqrt{1-k^2\sin(u)^2}~\mathrm{d}u=2e^t E(k)/\pi~.
\end{align*}
Hence, we can conclude that \(I_1(t)=I_2(t)\) for all \(t> 0\). For the case \(t=0\), the verification is straightforward.
\end{proof}

\bibliographystyle{alpha}
\bibliography{exotic}

\noindent{\sc François Fillastre}\\
\noindent{\sc  IMAG, Université de Montpellier, CNRS, Montpellier, France}

\noindent{\it Email address:} {\tt \href{mailto:francois.fillastre@umontpellier.fr}{francois.fillastre@umontpellier.fr}}

\vspace{1em}

\noindent{\sc Yusen Long}\\
\noindent{\sc Université Paris-Est Créteil, CNRS, LAMA UMR8050, F-94010 Créteil, France}

\noindent{\it Email address:} {\tt \href{mailto:yusen.long@u-pec.fr}{yusen.long@u-pec.fr}}

\vspace{1em}

\noindent{\sc David Xu}\\
\noindent{\sc School of Mathematics, Korea Institute for Advanced Study (KIAS), 02455 Seoul, Korea}

\noindent{\it Email address:} {\tt \href{mailto:davidxu@kias.re.kr}{davidxu@kias.re.kr}}

\end{document}